\documentclass[english]{smfart}
\usepackage{nearby-stokes}
\makeatletter
\let\oldsmf@enddoc\smf@enddoc
\renewcommand{\smf@enddoc}{\enlargethispage{3\baselineskip}\oldsmf@enddoc}
\makeatother
\begin{document}
\frontmatter
\title{Moderate and rapid-decay nearby cycles for holonomic $\mathcal D$-modules}

\author[C.~Sabbah]{Claude Sabbah}
\address{Centre de Mathématiques Laurent Schwartz (CMLS), CNRS, École polytechnique, Institut Polytechnique de Paris, Palaiseau, France}
\email{Claude.Sabbah@polytechnique.edu}
\urladdr{https://perso.pages.math.cnrs.fr/users/claude.sabbah}



\begin{abstract}
We introduce the notion of moderate and rapid decay nearby cycles relative to a holomorphic function $f$ for an arbitrary holonomic $\mathcal{D}$-module. They are proved to be $\mathbb{R}$-constructible complexes on the product of the special fiber of the function and the circle $S^1$ parametrizing the values of $f/|f|$. Duality properties are proved by B.\,Hepler and A.\,Hohl \cite{H-H25}, and we also prove them in special cases. Relations with the irregularity complexes as defined by Z.\,Mebkhout are given. Most of the arguments rely on the results of T.\,Mochizuki \cite{Mochizuki10}.
\end{abstract}

\maketitle
\tableofcontents
\mainmatter

\section{Introduction}\label{sec:intro}
Let $X$ be a complex manifold and let $f:X\to\CC$ be a holomorphic function. The construction by Grothendieck and Deligne of the nearby cycle functor \cite{Deligne73} associates to each sheaf $\cF$ of modules over a ring $A$, or each bounded complex of such sheaves, a~complex~$\psi_f\cF$ supported on the fiber $X_0:=f^{-1}(0)$, equipped with an automorphism~$\rT$, called the monodromy operator. In order to remain in the category of perverse sheaves when~$\cF$ is so and $A$ is a field, we consider instead the shifted complex $\ppsi_f\cF=\psi_f\cF[-1]$.

On the other hand, starting from a holonomic (left) $\cD_X$-module $\cM$, the theory of the Kashiwara-Malgrange $V$\nobreakdash-filtration enables one to define a holonomic $\cD_X$-module supported on~$X_0$, denoted here by $\psif\cM$, equipped with an automorphism $\rT$. Applying the de~Rham functor (with the shift by the dimension, as is usual in the theory of $\cD$-modules), we~obtain, according to a celebrated theorem of Kashiwara, a perverse complex with an automorphism $(\pDR\psif\cM,\rT)$.

If $\cM$ has \emph{regular singularities} along $X_0$, a theorem of Kashiwara and Malgrange identifies $(\pDR\psif\cM,\rT)$ with $(\ppsif\pDR\cM,\rT)$, constructed as in \cite{Deligne73} over the field~$\CC$. For a general holonomic $\cM$, the two complexes may differ.

If $X=\CC$, $f=\id$, and $\cM$ is $\cD_X$-holonomic with arbitrary singularities, the complexes $\pDR\psif\cM$ and $\ppsif\pDR\cM$ are finite-dimensional vector spaces (in degree zero) with an automorphism $\rT$, that one can equivalently regard as local systems of finite rank on the circle $S^1$. The relation between the two local systems is obtained by introducing the subsheaves of the latter consisting of sectorial germs of horizontal sections of~$\cM$ having moderate growth/rapid decay on $\partial\wt X(f)$. One recovers the former as the quotient sheaf of these two sheaves.

Our aim is to perform a similar construction in arbitrary dimension and for an~arbitrary holomorphic function $f$. To this end, we~introduce the following objects and notations.

Recall that we~set $X_0=f^{-1}(0)$. We~let $\varpif:\wt X(f)\to X$ be the real blowing-up of $f^{-1}(0)$. It~is a complex manifold with (a possibly singular) boundary. More precisely, one has $\partial\wt X(f):=\varpif^{-1}(X_0)\simeq \XOS$ (\cf Section~\ref{subsec:realblowupf} below). The universal covering space $X_0\times\RR\to X_0\times S^1$ is denoted by $\rho_0$, and the composition $\varpi_f\circ\rho_0:X_0\times\RR\to X_0$ by~$q_0$. We~also consider the translation $\tau:(x,t)\mto(x,t+1)$ on $X_0\times\RR$.

\subsubsection*{Topological nearby cycles}
Given a holonomic $\cD_X$-module $\cM$, we~will construct an $\RR$\nobreakdash-constructible complex $\ppsifstar\cM$ on $\partial\wt X(f)$, which we will compare with $(\ppsif\pDR\cM,\rT)$ in the following sense (\cf Section~\ref{subsec:pfq0-1}):
\begin{itemize}
\item
On the one hand, we~equip the pullback $q_0^{-1}\ppsif\pDR\cM$ to $X_0\times\RR$ with the isomorphism $\wh\rT:q_0^{-1}\ppsif\pDR\cM\isom\tau^{-1}q_0^{-1}\ppsif\pDR\cM$ by composing the automorphism $q_0^{-1}\rT$ with the identification $\tau^{-1}q_0^{-1}\ppsif\pDR\cM=q_0^{-1}\ppsif\pDR\cM$ (since $q_0\circ\tau=q_0$).
\item
On the other hand, the pullback $\rho_0^{-1}\ppsifstar\cM$ is naturally equipped with an isomorphism $\wh\rT:\rho_0^{-1}\ppsifstar\cM\isom\tau^{-1}\rho_0^{-1}\ppsifstar\cM$, since $\rho_0\circ\tau=\rho_0$.
\end{itemize}

\begin{proposition}\label{th:q0-1}
For a holonomic $\cD_X$-module $\cM$, or an object $\cM$ of the bounded derived category $\catD^\rb_\hol(\cD_X)$ of bounded complexes of $\cD_X$-modules with holonomic cohomology, there exists a functorial isomorphism between the objects $(q_0^{-1}\ppsif\pDR\cM,\wh\rT)$ and $(\rho_0^{-1}\ppsifstar\cM,\wh\rT)$. In particular, for any $x\in X_0$, the restriction of $\ppsifstar\cM$ to the circle $\{x\}\times S^1$ has locally constant cohomology.
\end{proposition}

\subsubsection*{Moderate and rapid-decay nearby cycles}
We refine the construction $\ppsifstar\cM$ by analogy with the one-dimensional case by~defining the complex of nearby cycles with moderate growth and that with rapid decay in the derived category $\catD^\rb_\Rc(\partial\wt X(f))$ of bounded complexes with $\RR$-constructible cohomology, denoted by $\wt\ppsif{}^\rmod\cM$ and $\wt\ppsif{}^\rrd\cM$ respectively, and we also obtain a complex $\wt\ppsif{}^\rmodrd\cM$ on $\partial\wt X(f)$. This construction can also be performed for $\cM$ in $\catD^\rb_\hol(\cD_X)$. We~obtain three distinguished triangles in the category $\catD^\rb_\Rc(\partial\wt X(f))$:
\begin{gather}
\wt\ppsif{}^\rmod\cM\to\ppsifstar\cM\to\wt\ppsif{}^\srmod\cM\To{+1},\label{eq:tridistmod}\\
\wt\ppsif{}^\rrd\cM\to\ppsifstar\cM\to\wt\ppsif{}^\srrd\cM\To{+1},\label{eq:tridistrd}\\
\wt\ppsif{}^\rrd\cM\to\wt\ppsif{}^\rmod\cM\to\wt\ppsif{}^\rmodrd\cM\To{+1},\label{eq:tridistmodrd}
\end{gather}
and by the octahedral axiom, a fourth one:
\begin{equation}
\wt\ppsif{}^\rmodrd\cM\to\wt\ppsif{}^\srrd\cM\to\wt\ppsif{}^\srmod\cM\To{+1}.
\end{equation}
If $\dim X=1$ and $\cM$ is a holonomic $\cD_X$-module, these triangles are in fact short exact sequences of sheaves (up to a shift) on $S^1$, as follows from the Hukuhara-Turrittin theorem (see \eg\cite[Th.\,1,\,p.\,205]{Malgrange91}) and the sheaves corresponding to $\ppsifstar\cM$ and $\wt\ppsif{}^\rmodrd\cM$ are local systems of finite-dimensional $\CC$\nobreakdash-vector spaces on $S^1$.

\subsubsection*{Formal nearby cycles}
By analogy to the case when $X$ is a curve, we~also prove in Section~\ref{subsec:pfq0-2}, for the complex $\wt\ppsif{}^\rmodrd\cM$ that fits into the distinguished triangle \eqref{eq:tridistmodrd}, a result analogous to that of Theorem~\ref{th:q0-1}.
\begin{itemize}
\item
We introduce the universal $\cD_X\lcr t\rcr\langle\partial_t\rangle$-module $\wh\Nils$ (corresponding to Nilsson class functions relative to the coordinate $t$ of $\CC$), which comes equipped with a monodromy operator $\rT$. We set
\[
\wh\cM^\nils:=\wh\Nils\otimes_{\cO_X\{t\}}\Dm\gamma_{f*}\cM|_{X\times\{0\}},
\]
where $\gamma_f:X\hto X\times\CC$ is the graph inclusion of $f$. Then the de~Rham complex $\pDR(\wh\cM^\nils)$, which is supported on $X_0\times\{0\}$ and which we simply regard as a complex on $X_0\times\{0\}$, is also equipped with such a monodromy operator; we then equip the pullback $q_0^{-1}\pDR (\wh\cM^\nils)$ to \hbox{$X_0\times\RR$} with the isomorphism $\wh\rT:q_0^{-1}\pDR(\wh\cM^\nils)\isom\tau^{-1}q_0^{-1}\pDR(\wh\cM^\nils)$ by composing the automorphism $q_0^{-1}\rT$ with the identification $\tau^{-1}q_0^{-1}\pDR(\wh\cM^\nils)=q_0^{-1}\pDR(\wh\cM^\nils)$.
\item
On the other hand, the pullback $\rho_0^{-1}\wt\ppsif{}^\rmodrd\cM$ is naturally equipped with an isomorphism $\wh\rT:\rho_0^{-1}\wt\ppsif{}^\rmodrd\cM\isom\tau^{-1}\rho_0^{-1}\wt\ppsif{}^\rmodrd\cM$, since $\rho_0\circ\tau=\rho_0$.
\end{itemize}

\begin{theoreme}\label{th:q0-2}
For a holonomic $\cD_X$-module $\cM$, or an object $\cM$ of the bounded derived category $\catD^\rb_\hol(\cD_X)$ of bounded complexes of $\cD_X$-modules with holonomic cohomology, there exists a functorial isomorphism between the objects $(q_0^{-1}\pDR(\wh\cM^\nils),\wh\rT)$ and $(\rho_0^{-1}\wt\ppsif{}^\rmodrd\cM,\wh\rT)$. In particular, for any $x\in X_0$, the restriction of $\wt\ppsif{}^\rmodrd\cM$ to the circle $\{x\}\times S^1$ has locally constant cohomology.
\end{theoreme}

\begin{remarque}
One might expect that $(\pDR(\wh\cM^\nils),\rT)\simeq(\pDR(\psi_f\cM),\rT)$ up to a shift. If $\cM$ has regular singularities, this statement is proved as an intermediate step in the comparison theorem for nearby cycles, due to Kashiwara and Malgrange (\cf\eg\cite[\S\S4.7--4.10]{M-S86b}). This expectation can be proved when $\dim X=1$, but we have not a proof in general.
\end{remarque}

\subsubsection*{The irregularity complex}
Let us also recall that, for a holonomic $\cD_X$-module $\cM$, the irregularity complexes
\begin{align*}
\pIrr_{X_0}\cM&:=\Irr_{X_0}\cM[\dim X]&\text{with }&\Irr_{X_0}\cM:=\bR\Gamma_{X_0}(\DR\cM(*X_0))\\
\tag*{and}\pIrr_{X_0}^*\cM&:=\Irr_{X_0}^*\cM[\dim X]&\text{with }&\Irr_{X_0}^*\cM:=i_f^{-1}\bR\cHom_{\cD_X}(\cM(*X_0),\cO_X)
\end{align*}
are perverse (\cf\cite[Th.\,3.5-2]{Mebkhout04}). For a holonomic $\cD_X$-module $\cM$, we~denote by~$\cM^\vee$ its dual, which is the left $\cD_X$-module associated to the right $\cD_X$-module $\cExt^{\dim X}_{\cD_X}(\cM,\cD_X)$.

\begin{proposition}\label{prop:perversite}
Let $\cM$ be a holonomic $\cD_X$-module. We~have functorial isomorphisms
\[
\bR\varpif{}_*\wt\ppsif{}^\rrd\cM[1]\simeq\pIrr^*_{X_0}\cM^\vee,\qquad\bR\varpif{}_*\wt\ppsif{}^\srmod\cM\simeq\pIrr_{X_0}\cM.
\]
\end{proposition}

(\Cf Section~\ref{pf:perversite} for the proof.) We conclude from the result of Mebkhout that the two complexes $\bR\varpif{}_*\wt\ppsif{}^\rrd\cM[1]$ and $\bR\varpif{}_*\wt\ppsif{}^\srmod\cM$ are perverse. Our main result is an analogous statement for objects on $\partial\wt X(f)$. Let us first adapt the notion of perversity to complexes on $\partial\wt X(f)$.

\subsubsection*{The $X_0$-support condition}
If $\cG$ is an $\RR$-constructible sheaf on $\partial\wt X(f)$, the \emph{$X_0$-support of $\cG$} is by definition the smallest closed complex analytic subset of $X_0$ containing the image under \hbox{$\varpi_f:\partial\wt X(f)\to X_0$} of the support of~$\cG$.

\begin{definition}[$X_0$-support condition]\label{def:X0support}
We~say that a bounded complex $\cG^\cbbullet$ with $\RR$-constructible cohomology on $\partial\wt X(f)$ satisfies the \emph{$X_0$-support condition} if
\[
\forall j,\quad\dim X_0\text{-}\Supp\cH^j\cG^\cbbullet\leq-j.
\]
\end{definition}

\begin{theoreme}\label{th:perversite}
Let $\cM$ be a holonomic $\cD_X$-module. Then the complexes $\ppsifstar\cM$, $\wt\ppsif{}^\srmod\cM$, $\wt\ppsif{}^\srrd\cM$, $\wt\ppsif{}^\rmod\cM[1]$, $\wt\ppsif{}^\rrd\cM[1]$ and $\wt\ppsif{}^\rmodrd\cM[1]$ on $\partial\wt X(f)$ have $\RR$\nobreakdash-con\-structible cohomology and satisfy the $X_0$-support condition.
\end{theoreme}

The $\RR$-constructibility part of this theorem follows from Remark~\ref{rem:Rconstr} and Theorem~\ref{th:DRmodRconst}. The $X_0$-support property is proved in Section~\ref{subsec:pfperversite}.

\begin{remarque}
One might wonder whether a stronger property occurs in Theorem~\ref{th:perversite}, namely, that $\wt\ppsif{}^\rmod\cM$ and $\wt\ppsif{}^\rrd\cM$ (hence $\wt\ppsif{}^\rmodrd\cM$ also) satisfy the $X_0$-support condition. This holds in the ``good case'' (\cf Section~\ref{subsec:goodperversite}).
\end{remarque}

\subsubsection*{Duality}
We denote by $\bD$ either the Poincaré-Verdier duality functor (\cf\eg\cite[Chap.\,3]{K-S90}), or the duality functor for bounded complexes of $\cD_X$-modules (so that, for a holonomic $\cD_X$\nobreakdash-mod\-ule, $\bD\cM\simeq\cM^\vee$). The following theorem, proved in Section~\ref{subsec:proofduality}, is a direct consequence of Conjecture~\ref{conj:DRdual} due to B.\,Hepler and A.\,Hohl \cite{H-H25}.

\begin{theoreme}[Behavior with respect to duality]\label{conj:dualite}
Let $\cM$ be an object of $\catD^\rb_\hol(\cD_X)$. The Poincaré-Verdier dual in $\catD^\rb_\Rc(\partial\wt X(f))$
\[
\bD\wt\ppsif{}^\srmod\cM\to\bD\ppsifstar\cM\to\bD\wt\ppsif{}^\rmod\cM\To{+1}
\]
of the distinguished triangle
\[
\wt\ppsif{}^\rmod\cM\to\ppsifstar\cM\to\wt\ppsif{}^\srmod\cM\To{+1}
\]
is functorially isomorphic to the distinguished triangle
\[
(\wt\ppsif{}^\rrd\bD\cM)[1]\to(\ppsifstar\bD\cM)[1]\to(\wt\ppsif{}^\srrd\bD\cM)[1]\To{+1}.
\]
A similar property holds after exchanging $\rmod$ and $\rrd$.
\end{theoreme}

We~can gather these two theorems as follows.

\begin{corollaire}\label{cor:dualite}
If $\cM$ is a holonomic $\cD_X$-module, the following pairs of dual objects satisfy the $X_0$-support condition:
\begin{itemize}
\item
$(\wt\ppsif{}^\srmod\cM,\bD\wt\ppsif{}^\srmod\cM)$, and $(\wt\ppsif{}^\srrd\cM,\bD\wt\ppsif{}^\srrd\cM)$,
\item
$(\wt\ppsif{}^\rmod\cM[1],\bD(\wt\ppsif{}^\rmod\cM[1]))$, and $(\wt\ppsif{}^\rrd\cM[1],\bD(\wt\ppsif{}^\rrd\cM[1]))$,
\end{itemize}
as well as the pair of \emph{shifted} dual objects
\[
(\ppsifstar\cM,(\bD\ppsifstar\cM)[-1])\quand(\wt\ppsif{}^\rmodrd\cM[1],\bD(\wt\ppsif{}^\rmodrd\cM)).\eqno\qed
\]
\end{corollaire}

\subsubsection*{Open question}
It would be natural to consider a refinement of the dichotomy moderate/rapid decay by considering precise exponential growth conditions. A comparison with the nearby slopes considered by J.-B.\,Teyssier in \cite{Teyssier16b} would then give an algebraic description of the jumping exponents occurring in the possible filtration by exponential growth.

\subsubsection*{Organization of the paper}
In Section~\ref{sec:realblowup} we recall basic constructions involving real blow-up spaces and in Section~\ref{sec:sheavesblup} we introduce the various sheaves of functions that we will need on these spaces. Their fundamental properties (mainly, flatness) have been proved in \cite{Mochizuki10} and we review them in Appendix~\ref{app:mochizuki}. To any holonomic $\cD_X$\nobreakdash-module we associate various de~Rham complexes on the real blow-up space $\wt X(f)$. We~examine their relations and prove their $\RR$-constructibility in Section~\ref{sec:dRham}. We~also recall duality properties for these de~Rham complexes, as proved in \cite{H-H25}, and which are shown to hold for meromorphic flat bundles in Appendix~\ref{app:C}, according to results in \cite{Mochizuki10}. For later use, we~also consider a relative statement in Section~\ref{subsec:noncar}. In Section~\ref{sec:nearby}, we~relate the complex $\ppsifstar\cM$ with the nearby cycle complex $\ppsi_f\pDR\cM$ equipped with its monodromy operator (\ie Theorem~\ref{th:q0-1} is proved), and a similar relation is established (\ie Theorem~\ref{th:q0-2} is proved) in Section~\ref{subsec:pfq0-2} for~$\wt\ppsif{}^\rmodrd\cM$ and $(\pDR\cM^\nils,\rT)$. The other theorems of Section~\ref{sec:intro} are also proved in Section~\ref{sec:pfintro}.

\subsubsection*{Acknowledgements}
Brian Hepler and Andreas Hohl solved the duality conjecture raised in an earlier version of this paper by means of the D'Agnolo--Kashiwara theory of enhanced ind-sheaves \cite{D-K13}. Many discussion with them were very enlightening. I also thank Takuro Mochizuki for many discussions regarding his work \cite{Mochizuki10}, which is the basis for the present one.

\section{Real blow-up spaces and their stratifications}\label{sec:realblowup}
\subsection{Real blow-up space along \texorpdfstring{$f=0$}{f0}}\label{subsec:realblowupf}
Let $f:X\to\CC$ be a holomorphic function on a complex manifold $X$ and set $X_0=f^{-1}(0)$, $X^*=X\moins X_0$. Recall (\cf\hbox{\cite[\S8.b]{Bibi10}}) that the real oriented blow-up $\wt X(f)$ of $X$ along $X_0$ is the closure in $X\times S^1$ of the graph of $f/|f|:X^*\to S^1$. The map $\varpif:\wt X(f)\to X$ is the restriction to $\wt X(f)$ of the first projection. We~have $\partial\wt X(f):=\varpif^{-1}(X_0)=X_0\times S^1$. As~a consequence, $\wt X(f)$ is the subset of the real analytic manifold $X\times S^1$ defined by the equation $f(x)-|f(x)|e^{i\theta}=0$, where $e^{i\theta}$ is the coordinate on $S^1$. This endows $\wt X(f)$ with the structure of a semi-analytic subset of $X\times S^1$.

We will use the following notation for the open and closed inclusions:
\[
\begin{aligned}
j_f:X^*&\hto X,\\
\wtj_f:X^*&\hto\wt X,
\end{aligned}
\qquad
\begin{aligned}
i_f:X_0&\hto X,\\
\wti_f:\partial\wt X(f)&\hto\wt X,
\end{aligned}
\]
so that $\varpi_f\circ\wtj_f=j_f$ and $\varpi_f\circ\wti_f=i_f\circ\varpi_f$.

Let us consider the graph inclusion $\gammaf:x\mto(x,f(x))$ in the following diagram:
\[
\xymatrix{
X\ar@/^2pc/[rr]_-{\id}\ar[d]_f\ar@{^{ (}->}[r]^-{\gammaf}&X\times\CC\ar[dl]^p\ar[r]^-r&X\\
\CC
}
\]
We then have the corresponding graph inclusion $\wt\gammaf$ in the corresponding diagram:
\begin{equation}\label{eq:realblowup}
\begin{array}{c}
\xymatrix{
\wt X(f)\ar@/^2.5pc/[rrr]_-{\varpif}\ar[d]_{\wt f}\ar@{^{ (}->}[r]^-{\wt\gammaf}&X\times\wt\CC\ar[dl]^{\wt p}\ar[r]^-\varpi&X\times\CC\ar[r]^-r&X\\
\wt\CC
}
\end{array}
\end{equation}
where $\wt\CC=S^1\times\RR_+$ is the oriented real blow-up of $\CC$ at the origin, and $\wt\gammaf(\wt X(f))$ is also identified with the closure of $\wt\gammaf(X^*)=\gammaf(X^*)$ in $X\times\wt\CC$. We~thus have $\wt\gammaf=(\varpif,\wt f)$.

We also note that the diagram
\begin{equation}\label{eq:diagcart}
\begin{array}{c}
\xymatrix{
\wt X(f)\ar@{^{ (}->}[r]^-{\wt\gammaf}\ar[d]_{\varpi_f}\ar@{}[rd]|\square&X\times\wt\CC\ar[d]^{\varpi}\\
X\ar@{^{ (}->}[r]^-{\gammaf}&X\times\CC
}
\end{array}
\end{equation}
is Cartesian (this is clear away from $(X_0,X\times\{0\})$ and on $(X_0,X\times\{0\})$).

For every morphism $\pi:Y\to X$ of complex manifolds, setting $g=f\circ\pi$, we~have a natural morphism $\wt\pi:\wt Y(g)\to\wt X(f)$ induced by the real-analytic map $\pi\times\id_{S^1}$.

We can make the construction of $\wt X(f)$ more global, associating it with the divisor defined by $f$. Let $D$ be any divisor in $X$ and let $L(D)$ be the bundle associated with~$D$, equipped with a section $f:\cO_X\to\cO_X(D)$ defining the divisor $D$ as $f^*(0)$. Let $S^1_D$ be the associated $S^1$-bundle on $X$. We~thus have a section $f/|f|:X\moins D\to S^1_D$, and the closure of its image is ``the'' real blow-up space of $X$ along $D$. Locally, it~is defined by a real-analytic equation in $S^1_D$ as above, which endows it with the structure of a semi-analytic subset of the real analytic manifold $S^1_D$. If we change the section (by~a~unit $u\in\Gamma(X,\cO_X^*)$), then the map $u/|u|:S^1_D\to S^1_D$ induces a real-analytic isomorphism between the two real blow-up spaces.

\subsection{The case of a normal crossing divisor}\label{subsec:realblowupD}
Let $Y$ be a complex manifold equipped with a normal crossing divisor $D$ with smooth components $D_i$ ($i\in I$). We~choose sections $f_i$ of $L(D_i)$.

The real blow-up space $\wt Y(D_{i\in I})$ of $Y$ along the components $D_i$ of $D$, which we simply denote here and throughout the rest of this article by $\wt Y(D)$, is~the closure in the fiber product $\bigtimes_{Y,i\in I}S^1_{D_i}$ of the image of $Y\moins D$ under the section $(f_i/|f_i|)_{i\in I}$. A local computation shows that $\wt Y(D)=\bigtimes_{Y,i\in I}\wt Y(D_i)$. Therefore, $\wt Y(D)$ is a semi-analytic subset of the real manifold $\bigtimes_{Y,i\in I}S^1_{D_i}$, and this structure does not depend on the choices made (sections $f_i$, or the ordering of $I$ to define the fiber product). Moreover, $\wt Y(D)$ is a complex manifold with a topologically smooth boundary, which is locally real-analytic isomorphic to $(S^1)^\ell\times \nobreak(\RR_+)^\ell\times\nobreak\CC^{n-\ell}$, in a neighborhood of any point in the intersection of a collection of exactly~$\ell$ components of $D$.

Assume now that $g:Y\to\CC$ is a holomorphic function such that $g^{-1}(0)=D$, with~$D$ as above. Then we have a factorization of $\varpi_D$:
\[
\wt Y(D)\To{\varpi_{D,g}}\wt Y(g)\To{\varpig}Y.
\]
This means that, in local coordinates, we~can obtain the argument of $g$ from the arguments of the coordinates defining the components of $D$.

\Subsection{Semi-analytic stratifications attached to a stratified \texorpdfstring{$\ccI$}{I}-covering}\label{subsec:Icovering}

\subsubsection*{Local study}
Let us keep the setting as in Section~\ref{subsec:realblowupD} and fix local coordinates $(y_1,\dots,y_n)$ on $Y$ centered at a point of $D$ such that $D=\{y_1\cdots y_\ell=0\}$ in a neighborhood of this point, which we still denote by $Y$. Let $\bmd=(d_1,\dots,d_\ell,1,\dots,1)$ be an $n$-multi-index consisting of positive integers and let
\begin{align*}
Y_{\bmd}&\To{\rho_{\bmd}} Y\\
(y'_1,\dots,y'_\ell,y'_{>\ell})&\Mto{\hphantom{\rho_{\bmd}}}(y_1^{\prime d_1},\dots,y_\ell^{\prime d_\ell},y'_{\sssup\ell})=(y^{\prime\bmd},y'_{\sssup\ell})
\end{align*}
be the corresponding ramified covering. Then $D\simeq D_{\bmd}:=\rho_{\bmd}^{-1}(D)=\{y'_1\cdots y'_\ell=0\}$. We~have $\wt Y_{\bmd}\simeq(S^1)^\ell\times\RR_+^\ell\times\CC^{n-\ell}$ with coordinates $(e^{i\theta'},r',y'_{\sssup\ell})$, $\partial\wt Y_{\bmd}$ is defined by $\prod_{i=1}^\ell r'_i=0$, and the map $\rho_{\bmd}$ can be lifted to the map $\wt\rho_{\bmd}:\wt Y_{\bmd}\to\wt Y$ given in coordinates by $(e^{i\theta'},r',y'_{\sssup\ell})\mto (e^{i\bmd\theta'},r^{\prime\bmd},y'_{\sssup\ell})$.

Let $\varphi\in\cO_{Y_{\bmd}}(*D_{\bmd})/\cO_{Y_{\bmd}}$ be a \emph{purely monomial} polar part of meromorphic function, that is, $\varphi=u(y')y^{\prime-\bmm}$ for some $\bmm\in\NN^\ell$ and $u'$ invertible or identically zero. It~defines a closed semi-analytic subset of $\partial\wt Y_{\bmd}$ with equation
\[
\arg u(y')-\sum_im_i\theta'_i=\pm\pi/2
\]
(this is the product of $2\gcd(\bmm)$ disjoint subtori $(S^1)^{\ell-1}$ with $\partial\RR_+^\ell\times\nobreak\CC^{n-\ell}$). The union of its images under the group of deck transformations is the pullback under $\rho_{\bmd}$ of a closed semi-analytic subset of $\partial\wt Y$ that is uniquely determined by $\varphi$.

Similarly, any finite family $\Phi_{\bmd}\subset\cO_{Y_{\bmd}}(*D_{\bmd})/\cO_{Y_{\bmd}}$ whose elements are purely monomial defines a closed semi-analytic subset of $\partial\wt Y$ (the union of those defined by each member of the family).

\subsubsection*{Global study}
The notion of a subset $\Phi_{\bmd}$ can be globalized along $D$ as the notion of a stratified $\ccI$-covering $\wt\Sigma$ (\cf\cite[Def.\,1.46 \& \S9.c]{Bibi10}). It~corresponds to that of a system of irregular values as defined in \cite[Def.\,2.4.2]{Mochizuki08}. One usually adds a goodness condition on $\Phi_{\bmd}\cup\{0\}$ (\cf\cite[Def.\,2.1.2]{Mochizuki08}, \cite[Def.\,9.12]{Bibi10}), which ensures the pure monomiality of its local sections. Recall that a meromorphic flat bundle~$\cM$ on~$Y$ with poles along~$D$ and having a good formal structure along $D$ determines a good stratified $\ccI$-covering $\wt\Sigma(\cM)$ of $\partial\wt Y(D)$ (\cf\loccit). As~the locally defined closed semi-analytic subset above depends only on the family $\Phi_{\bmd}$, we~obtain:

\begin{lemme}
Let $\wt\Sigma$ be a good stratified $\ccI$-covering of $\partial\wt Y(D)$. Then the locally defined semi-analytic subsets of $\partial\wt Y(D)$ ($y\in D$) glue together and define a closed semi-analytic subset of $\partial\wt Y(D)$.\qed
\end{lemme}

There exist then semi-analytic stratifications $\partial\wt\cY$ of $\partial\wt Y(D)$ compatible with this closed semi-analytic set. We~will consider any such stratification.

\section{Sheaves on the real blow-up spaces}\label{sec:sheavesblup}
In this section, we~recall various results from \cite{Mochizuki10}. We~add some easy complements, the proofs of which are given in Appendix~\ref{app:mochizuki}, following the same lines as in \loccit

\subsection{Sheaves of functions on the real blow-up space along \texorpdfstring{$f=0$}{f0}}\label{subsec:sheavesblup}
We keep the notation of Section~\ref{subsec:realblowupf} and we implicitly refer to Diagram \eqref{eq:realblowup}. We~will be mainly interested in the following two sheaves of functions on $\wt X(f)$, whose restrictions to $X^*$ are equal to $\cO_{X^*}$:
\begin{itemize}
\item
$\cA_{\wt X(f)}^{\rmod X_0}$, which we simply denote by $\cA_{\wt X(f)}^\rmod$, is~the sheaf of functions which are holomorphic on $X^*$ and have moderate growth along $\partial\wt X(f)$,
\item
$\cA_{\wt X(f)}^{\rrd X_0}$, which we simply denote by $\cA_{\wt X(f)}^\rrd$, is~the sheaf of functions that are holomorphic on $X^*$ and have rapid decay along $\partial\wt X(f)$.
\end{itemize}
\begin{itemize}
\item
We also use the notation $\cA_{\wt X(f)}^*$ for $\wtjf{}_*\cO_{X^*}$. We~thus have natural inclusions\vspace*{-3pt}
\[
\cA_{\wt X(f)}^\rrd\subset\cA_{\wt X(f)}^\rmod\subset\cA_{\wt X(f)}^*:=\wtjf{}_*\cO_{X^*}.
\]
\item
Furthermore, we set\vspace*{-3pt}
\begin{align*}
\cA_{\wt X(f)}^\srmod&:=\cA_{\wt X(f)}^*/\cA_{\wt X(f)}^\rmod,\\
\cA_{\wt X(f)}^\srrd&:=\cA_{\wt X(f)}^*/\cA_{\wt X(f)}^\rrd,\\
\cA_{\wt X(f)}^\rmodrd&:=\cA_{\wt X(f)}^\rmod/\cA_{\wt X(f)}^\rrd\simeq\textstyle\ker\bigl[\cA_{\wt X(f)}^\srrd\onto\cA_{\wt X(f)}^\srmod\bigr],
\end{align*}
which are sheaves supported on $\partial\wt X(f)$.
\end{itemize}

\begin{notation}\label{nota:cB}
We will use the notation $\cA_{\wt X(f)}^\Star$ to denote any of the previous sheaves, with $\Star\in\{*,\rmod,\rrd,\ssup\rmod,\ssup\rrd,\rmodrd\}$.
\end{notation}

Similarly, we~will consider the special case of a projection, \ie the space $X\times\CC$ with divisor $X\times\{0\}$ and real blown-up space $X\times\wt\CC$, and the corresponding sheaves~$\cA^\Star_{X\times\wt\CC}$. Let us first note the following properties.
\numero\label{subsec:sheavesblupa}
Multiplication by $f$ is invertible on $\cA^\Star_{\wt X(f)}$ (this is obvious).
\numero\label{subsec:sheavesblup0}
We have $\bR\wtjf{}_*\cO_{X^*}=\wtjf{}_*\cO_{X^*}$.

Indeed, identifying $\wt X(f)$ with its image under $\wt\gammaf$, each point of $\partial\wt X(f)\subset X\times S^1\subset X\times\wt\CC$ has a fundamental system of open neighborhoods $\wt V=(U\times\wt\Sigma)\cap \wt X(f)$, where $U\subset X$ is Stein and $\wt\Sigma\subset\wt\CC$ is defined by $\theta\in(\theta_o-\epsilon,\theta_o+\epsilon)$, if $e^{i\theta}$ is the coordinate on $S^1$ and $\wt\CC$ is identified with $S^1\times\RR_+$. Set $\Sigma^*=\wt\Sigma\cap\CC^*$. Note that $(U^*\times\Sigma^*)$ is a Stein open set, since $U^*$ is the complement of a hypersurface in $U$ and $\Sigma^*$ is convex. Then $\wt V\cap X^*=(U^*\times\Sigma^*)\cap \wt X(f)=(U^*\times\Sigma^*)\cap\gammaf(X^*)$ is the intersection of a Stein open subset of $X^*\times\CC^*$ with the closed submanifold $\gammaf(X^*)$, and is hence also Stein, and Cartan--Serre's Theorem B yields Assertion~\ref{subsec:sheavesblup0}.

Let us recall basic results concerning these sheaves in the present setting, proved in \cite[Th.\,4.1.5]{Mochizuki10} in greater generality. (See Appendix~\ref{app:mochizuki} for details.)
\numero\label{subsec:sheavesblup1}
The sheaves $\cA^\Star_{\wt X(f)}$ are $\varpif^{-1}\cO_X$-flat.

\numero\label{subsec:sheavesblup2}
$\varpi_f^{-1}\cO_X\otimes^{\bL}_{\wt\gamma_f^{-1}\varpi^{-1}\cO_{X\times\CC}}\wt\gamma_f^{-1}\cA^\Star_{X\times\wt\CC}=\varpi_f^{-1}\cO_X\otimes_{\wt\gamma_f^{-1}\varpi^{-1}\cO_{X\times\CC}}\wt\gamma_f^{-1}\cA^\Star_{X\times\wt\CC}\simeq\cA^\Star_{\wt X(f)}$.

Since $\varpif^{-1}\cD_X$ acts in a natural way on $\cA^\Star_{\wt X(f)}$ (\ie the condition $\Star$ is preserved by derivation), the sheaf
\[
\cD^\Star_{\wt X(f)}:=\cA^\Star_{\wt X(f)}\otimes_{\varpif^{-1}\cO_X}\varpif^{-1}\cD_X=\varpif^{-1}\cD_X\otimes_{\varpif^{-1}\cO_X}\cA^\Star_{\wt X(f)}
\]
is a sheaf of rings on $\wt X(f)$. Any $\cD_X$-module $\cM$ gives rise to a $\cD^\Star_{\wt X(f)}$-module
\[
\varpi_f^\Star\cM:=\cD^\Star_{\wt X(f)}\otimes_{\varpi_f^{-1}\cD_X}\varpi_f^{-1}\cM=\cA^\Star_{\wt X(f)}\otimes_{\varpi_f^{-1}\cO_X}\varpi_f^{-1}\cM.
\]
As \ref{subsec:sheavesblup2} is compatible with the $\cD$-module structure, it also reads:

\numero\label{subsec:sheavesblup2b}
$\varpi_f^{-1}\cD_{X\To{\gamma_f} X\times\CC}\otimes^{\bL}_{\varpi^{-1}\cD_{X\times\CC}}\cA^\Star_{X\times\wt\CC}\simeq\cA^\Star_{\wt X(f)}$ as a left $\varpi_f^{-1}\cD_X$-module.

\numero\label{subsec:sheavesblup2c}
Flatness of $\cA^\Star_{\wt X(f)}$ over $\varpi_f^{-1}\cO_X$ immediately implies that we have exact sequences
\begin{gather}
0\to\varpi_f^\rmod\cM\to\varpi_f^*\cM\to\varpi^\srmod\cM\to0,\tag{\ref{subsec:sheavesblup2c}${*}$}\label{eq:tridistDmod}\\
0\to\varpi_f^\rrd\cM\to\varpi_f^*\cM\to\varpi_f^\srrd\cM\to0,\tag{\ref{subsec:sheavesblup2c}${*}{*}$}\\
0\to\varpi_f^\rrd\cM\to\varpi_f^\rmod\cM\to\varpi_f^\rmodrd\cM\to0.\tag{\ref{subsec:sheavesblup2c}${*}{*}{*}$}\label{eq:tridistDmodrd}
\end{gather}

On the other hand, we~have (\cf\eg\cite[\S II.1.1]{Bibi97}):
\numero\label{subsec:sheavesblup3}
$\bR\varpi_*\cA_{X\times\wt\CC}^\rmod=\varpi_*\cA_{X\times\wt\CC}^\rmod\simeq\cO_{X\times\CC}(*(X\times0))$,
\numero\label{subsec:sheavesblup4}
$\bR\varpi_*\cA_{X\times\wt\CC}^\rrd\simeq\{0\to\cO_{X\times\CC}\to\cO_{\wh{X\times\CC|X\times0}}\to0\}$, where the latter sheaf is the formal completion of $\cO_{X\times\CC}$ along $X\times0$ (it is zero on $X\times\CC^*$).

We then deduce:
\numero\label{subsec:sheavesblup5}
$\bR\varpif{}_*\cA_{\wt X(f)}^\rmod=\varpif{}_*\cA_{\wt X(f)}^\rmod\simeq\cO_X(*X_0)$. Indeed,
\tagdroite
\begin{align*}
\bR\gammaf{}_*\bR\varpif{}_*&\cA_{\wt X(f)}^\rmod\simeq\bR\varpi_*\bR\wt\gammaf{}_*\cA_{\wt X(f)}^\rmod\\
&\simeq\bR\varpi_*(\varpi^{-1}\cO_{\gammaf(X)}\otimes_{\varpi^{-1}\cO_{X\times\CC}}\cA_{X\times\wt\CC}^\rmod)\quad\text{(by \ref{subsec:sheavesblup2})}\\
&=\bR\varpi_*(\varpi^{-1}\cO_{\gammaf(X)}\overset{\bL}\otimes_{\varpi^{-1}\cO_{X\times\CC}}\cA_{X\times\wt\CC}^\rmod)\quad\text{(by \ref{subsec:sheavesblup1})}\\
&\simeq\cO_{\gammaf(X)}\overset{\bL}\otimes_{\cO_{X\times\CC}}\bR\varpi_*\cA_{X\times\wt\CC}^\rmod\quad\text{(projection formula \cite[Prop.\,2.6.6]{K-S90})}\\
&\simeq\cO_{\gammaf(X)}\otimes_{\cO_{X\times\CC}}\cO_{X\times\CC}(*(X\times0))\simeq\bR\gammaf{}_*\cO_X(*X_0)\quad\text{(by \ref{subsec:sheavesblup3})}.\tag*{\qed}
\end{align*}
\taggauche
\numero\label{subsec:sheavesblup6}
$\bR\varpif{}_*\cA_{\wt X(f)}^\rrd\simeq\{0\to\cO_X\to\cO_{\wh{X|X_0}}\to0\}$ or, equivalently,
\[\textstyle
\bR\varpif{}_*\cA_{\wt X(f)}^\rrd\simeq\{0\to\cO_X(*X_0)\to\cO_{\wh{X|X_0}}(*X_0)\to0\}.
\]
Indeed, we~have similarly
\[\textstyle
\bR\gammaf{}_*\bR\varpif{}_*\cA_{\wt X(f)}^\rrd\simeq\cO_{\gammaf(X)}\overset{\bL}\otimes_{\cO_{X\times\CC}}\{0\to\cO_{X\times\CC}\to\cO_{\wh{X\times\CC|X\times0}}\to0\}
\]
and by flatness of $\cO_{\wh{X\times\CC|X\times0}}$ over $\cO_{X\times\CC|X\times0}$, the assertion is reduced to the identification (\cf\eg\cite[Cor.\,II.2]{Serre65})
\[
\cO_{\gammaf(X)}\otimes_{\cO_{X\times\CC}}\cO_{\wh{X\times\CC|X\times0}}\simeq\gammaf{}_*\cO_{\wh{X|X_0}}.\eqno\qed
\]
We set
\begin{equation}\label{eq:QX0}
\cQ_{X_0}=\iif^{-1}(\cO_{\wh{X|X_0}}/\cO_X)=\iif^{-1}\cO_{\wh{X|X_0}}/\iif^{-1}\cO_X
\end{equation}
(notation of \cite{Mebkhout90}). Then
\[
\iif^{-1}\bR\varpif{}_*\cA_{\wt X(f)}^\rrd\simeq\cQ_{X_0}[-1].\eqno\qed
\]

\numero\label{subsec:sheavesblup7}
We also conclude that
\[
\bR\varpif{}_*\cA_{\wt X(f)}^\rmodrd\simeq\cO_{\wh{X|X_0}}(*X_0).
\]

For every projective morphism $\pi:Y\to X$, setting $g=f\circ\pi$, we~also have, according to \cite[Th.\,4.1.5]{Mochizuki10}:

\numero\label{eq:pushpiA}
If $\cM$ is an inductive limit of coherent $\cO_Y$-modules, then
\[
\bR\wt\pi_*(\cA^\Star_{\wt Y(g)}\overset{\bL}\otimes_{\varpi_g^{-1}\cO_Y}\varpi_g^{-1}\cM)\simeq\cA^\Star_{\wt X(f)}\overset{\bL}\otimes_{\varpif^{-1}\cO_X}\varpif^{-1}\bR\pi_*\cM.
\]
Indeed, by the flatness property \ref{subsec:sheavesblup1}, we~can just consider the non-derived tensor product and argue with $R^k\wt\pi_*,R^k\pi_*$ for each $k$. As~$\wt\pi,\pi$ are proper, taking inductive limits commutes with taking $R^k\wt\pi_*,R^k\pi_*$, and we are reduced to considering the case where~$\cM$ is $\cO_Y$-coherent, which is precisely \cite[Th.\,4.1.5]{Mochizuki10}.

\numero\label{eq:pusheA}
For example, if $e:Y\to X$ is a proper modification that is an isomorphism over $X\moins X_0$, we~deduce from \ref{eq:pushpiA}:
\[
\cA^\Star_{\wt X(f)}=\wt e_*\cA^\Star_{\wt Y(g)}=\bR\wt e_*\cA^\Star_{\wt Y(g)}.
\]
Indeed, we~first note that $\bR e_*\cO_Y(*Y_0)$ has $\cO_X(*X_0)$-coherent cohomology, and since $R^ke_*\cO_Y(*Y_0)$ is supported on $X_0$ for $k\geq1$, it~follows that $\bR e_*\cO_Y(*Y_0)=e_*\cO_Y(*Y_0)=\cO_X(*X_0)$. Since $\cA^\Star_{\wt Y(g)}$, \resp $\cA^\Star_{\wt X(f)}$, is~an $\cO_Y(*Y_0)$-module, \resp an $\cO_X(*X_0)$-module, \ref{eq:pushpiA} reads
\begin{multline*}
\bR\wt e_*(\cA^\Star_{\wt Y(g)}\overset{\bL}\otimes_{\varpi_g^{-1}\cO_Y(*Y_0)}\varpi_g^{-1}(\cO_Y(*Y_0)\otimes_{\cO_Y}\cM))\\
\simeq\cA^\Star_{\wt X(f)}\overset{\bL}\otimes_{\varpif^{-1}\cO_X(*X_0)}\varpif^{-1}(\cO_X(*X_0)\overset{\bL}\otimes_{\cO_X}\bR e_*\cM),
\end{multline*}
and the latter term is identified with
\[
\cA^\Star_{\wt X(f)}\overset{\bL}\otimes_{\varpif^{-1}\cO_X(*X_0)}\varpif^{-1}\bR e_*(\cO_Y(*Y_0)\overset{\bL}\otimes_{\cO_Y}\cM).
\]
Applying this to $\cM=\cO_Y$, we~find
\tagdroite
\begin{align*}
\bR\wt e_*\cA^\Star_{\wt Y(g)}&\simeq\cA^\Star_{\wt X(f)}\overset{\bL}\otimes_{\varpif^{-1}\cO_X(*X_0)}\varpif^{-1}\bR e_*\cO_Y(*Y_0)\\
&\simeq\cA^\Star_{\wt X(f)}\overset{\bL}\otimes_{\varpif^{-1}\cO_X(*X_0)}\varpif^{-1}\cO_X(*X_0)\quad\text{by the preliminary remark}\\
&\simeq\cA^\Star_{\wt X(f)}.\tag*{\qed}
\end{align*}
\taggauche

\subsection{Nilsson class functions}
We work on $X\times\CC$ with coordinate $t$ on $\CC$ and on the real blow-up space $X\times\wt\CC$. Recall that the blowing-up map is denoted by $\varpi$. Let $\cA_{X\times\wt\CC}$ be the subsheaf of $\cC^\infty_{X\times\wt\CC}$ consisting of functions killed by $\ov\partial_x$ and $t\ov\partial_t$. With respect to $\cA_{X\times\wt\CC}^\rrd$ and $\cA_{X\times\wt\CC}^\rmod$ defined in Section~\ref{subsec:sheavesblup}, we~have the inclusions
\[
\cA_{X\times\wt\CC}^\rrd\subset\cA_{X\times\wt\CC}\subset\cA_{X\times\wt\CC}^\rmod.
\]
The restriction of each of these sheaves to $X\times\CC^*$ is equal to $\cO_{X\times\CC^*}$. There is an exact sequence (\cf\eg\cite[Chap.\,II, Prop.\,1.1.16]{Bibi97})
\begin{equation}\label{eq:suiteexrdformel}
0\to\cA_{X\times\wt\CC}^\rrd\to\cA_{X\times\wt\CC}\to\varpi^{-1}\cO_X\lcr t\rcr\to0.
\end{equation}
We consider Nilsson class functions in two ways:
\begin{enumerate}
\item
Algebraically, we~consider the $\cO_X[t]$-module
\[
\Nils=\bigoplus_{\alpha\in\CC}\Nils_\alpha=\bigoplus_{\substack{\alpha\in\CC\\p\in\NN}}\cO_X\otimes e_{\alpha,p}
\]
with the $\cD_X\otimes\Clt$-module structure given by (identifying $1\otimes e_{\alpha,p}$ with $e_{\alpha,p}$)
\[
t\cdot e_{\alpha,p}=e_{\alpha+1,p},\quad \partial_te_{\alpha,p}=(\alpha+1) e_{\alpha-1,p}+e_{\alpha-1,p-1},\quad\partial_xe_{\alpha,p}=0,
\]
where we adopt the convention that $e_{\alpha,-1}=0$ for any $\alpha\in\CC$. In other words, $e_{\alpha,p}$~behaves like the multi-valued holomorphic function $t^{\alpha+1}(\log t)^p/p!$. We~note that $\Nils$ is a $\cD_X\otimes\Cltm$-module, that is, $t$ acts in an invertible way on $\Nils$. The monodromy operator on $\Nils$ is the operator $\rT_\Nils$ acting as $\exp(2\pii t\partial_t)$.

\item
We also consider the corresponding formal sheaf
\[
\wh\Nils=\cO_X\lcr t\rcr\otimes_{\cO_X[t]}\Nils,
\]
which we also regard as a sheaf of $\cO_X\{t\}$-modules.

\item
We denote by $\cA^\nils_{X\times\wt\CC}$ the subsheaf of $\cA^\rmod_{X\times\wt\CC}$ generated over $\cA_{X\times\wt\CC}$ by the local determinations of the functions $t^{\alpha+1}(\log t)^p/p!$ ($\alpha\in\CC,\,p\in\NN$).
\end{enumerate}

Let $\rho:\RR\to S^1$ denote the universal covering $\theta\mto\rme^{2\pii\theta}$ and also the corresponding covering $X\times(\RR\times\RR_+)\to X\times\wt\CC$. We have a natural inclusion
\begin{align*}
(\varpi\circ\rho)^{-1}\Nils&\hto\rho^{-1}\cA_{X\times\wt\CC}^\rmod\\
e_{\alpha,p}&\mto r^{\alpha+1}\rme^{2\pii\alpha\theta}\frac{(\log r+2\pii\theta)^p}{p!}
\end{align*}
inducing an isomorphism
\[
\bigoplus_{\alpha,p}(\rho^{-1}\cA_{X\times\wt\CC}\otimes e_{\alpha,p})\isom\rho^{-1}\cA_{X\times\wt\CC}^\nils.
\]
The exact sequence \eqref{eq:suiteexrdformel} therefore yields the exact sequence
\begin{equation}\label{eq:suiteexnilsformel}
0\to\rho^{-1}\cA_{X\times\wt\CC}^\rrd\to\rho^{-1}\cA_{X\times\wt\CC}^\nils\to(\varpi\circ\rho)^{-1}\wh\Nils\to0.
\end{equation}

Recall that $\gamma_f:X\hto X\times\CC$ denotes the graph inclusion of $f$. We set
\[
\cA_{\wt X(f)}^\nils:=\varpi^{-1}\cO_{\gamma_f(X)}\otimes_{\varpi^{-1}\cO_{X\times\CC}}\cA_{X\times\wt\CC}^\nils.
\]
This is indeed a sheaf on $\wt X(f)$, since the diagram \eqref{eq:diagcart} is Cartesian. It~is denoted by $\cA_{X,f}^\mathrm{nil}$ in \cite[Prop.\,4.5.14]{Mochizuki10}.

On the other hand, if $(Y,D)$ is a normal crossing pair, the sheaf $\cA_{\wt Y(D)}^\nils$ is defined similarly (\cf\cite[\S4.1.4.2]{Mochizuki10}). Let $e:Y\to X$ be a projective modification such that the divisor $D_g$ of $g=f\circ e$ has normal crossings. There exists a natural morphism $\wt e:\wt Y(D_g)\to\wt X(f)$.

\begin{theoreme}[{\cite[Th.\,4.5.12]{Mochizuki10}}]
There is an isomorphism
\[
\bR \wt e_*\cA_{\wt Y(D_g)}^\nils\isom\cA_{\wt X(f)}^\nils.
\]
\end{theoreme}

\subsection{The case of a normal crossing divisor}\label{subsec:ncdA}
The sheaves $\cA^\Star_{\wt Y(D)}$ are defined on $\wt Y(D)$ similarly to the case of a smooth divisor (Section~\ref{subsec:sheavesblup}). The following results hold, according to \cite[Th.\,4.1.5, Prop.\,4.2.4, Th.\,4.5.1]{Mochizuki10}.

\numero\label{subsec:sheavesblupD1}
The sheaves $\cA^\Star_{\wt Y(D)}$ are $\varpi_D^{-1}\cO_Y$-flat.
\numero\label{subsec:sheavesblupD2}
$\bR\varpi_{D,g*}\cA^\Star_{\wt Y(D)}=\varpi_{D,g*}\cA^\Star_{\wt Y(D)}=\cA^\Star_{\wt Y(g)}$.

\subsection{Localization and formalization of $\cD_X$-modules}
For a coherent $\cD_X$\nobreakdash-mod\-ule $\cM$, the localized $\cD_X$-module $\cM(*X_0)$ is defined as $\cO_X(*X_0)\otimes_{\cO_X}\cM$. On the other hand, the formalized $\cD_X$-module $\cM_{\wh{X|X_0}}$ along $X_0$ is defined as $\cO_{\wh{X|X_0}}\otimes_{\cO_X}\cM$. There are natural morphisms
\[
\cM\to\cM(*X_0)\quand \cM\to\cM_{\wh{X|X_0}}.
\]
Let us denote by $\cQ_{\wh{X|X_0}}$ the complex $\cO_X\to\cO_{\wh{X|X_0}}$ with terms in degrees~$0$ and~$1$ respectively. With the previous notation \eqref{eq:QX0}, we~have $i_f^{-1}\cH^1\cQ_{\wh{X|X_0}}=\cQ_{X_0}$. The sheaf-theoretic restriction of $\cQ_{\wh{X|X_0}}$ to $X^*$ is $\cO_{X^*}$ and that to $X_0$ is $\cO_{\wh{X|X_0}}/\cO_X[-1]$. This complex is isomorphic to the complex $\cO_X(*X_0)\to\cO_{\wh{X|X_0}}(*X_0)$. For a bounded complex of $\cD_X$-modules, there is a distinguished triangle in $\catD^\rb(\cD_X)$:
\begin{equation}\label{eq:QMM}
\cQ_{\wh{X|X_0}}\otimes^{\bL}_{\cO_X}\cM\to\cM(*X_0)\to\cM_{\wh{X|X_0}}(*X_0)\To{+1}.
\end{equation}

According to Properties \ref{subsec:sheavesblup5}--\ref{subsec:sheavesblup7} and the projection formula for $\varpi_f$, we~deduce:

\begin{proposition}
Let $\cM$ be a $\cD_X$-module. The distinguished triangle associated to the pushforward by $\varpi_f$ of the exact sequence \eqref{eq:tridistDmodrd} is isomorphic to the distinguished triangle \eqref{eq:QMM}.\qed
\end{proposition}

\section{The moderate and rapid-decay de~Rham complexes on \texorpdfstring{$\protect\wt{X}(f)$}{Xf}}\label{sec:dRham}

\Subsection{Moderate and rapid-decay de~Rham complexes for \texorpdfstring{$\cD_X$}{DX}-modules}
We still use Notation~\ref{nota:cB}. Let $\cM$ be a $\cD_X$-module. Since $\cA^\Star_{\wt X(f)}$ is a left $\varpif^{-1}\cD_X$\nobreakdash-mod\-ule, we~can set
\[
\DR_{\wt X(f)}^\Star\cM:=\bigl(\cA^\Star_{\wt X(f)}\otimes_{\varpif^{-1}\cO_X}\varpif^{-1}(\Omega_X^\cbbullet\otimes\cM),\nabla\bigr)
\]
and
\[
\pDR_{\wt X(f)}^\Star\cM:=\DR_{\wt X(f)}^\Star\cM[\dim X].
\]
We can replace $\Omega_X^\cbbullet$ with $\Omega_X^\cbbullet(*X_0)$ since $f$ is invertible on $\cA^\Star_{\wt X(f)}$. Recall that the Spencer complex $\Sp(\cD_X)$ is a resolution of $\cO_X$ by locally free left $\cD_X$-modules. Then the $\Star$-Spencer complex $\Sp(\cD_{\wt X(f)}^\Star):=\cA_{\wt X(f)}^\Star\otimes_{\varpif^{-1}\cO_X}\varpif^{-1}\Sp(\cD_X)$ is a resolution of $\cA_{\wt X(f)}^\Star$ by locally free left $\cD_{\wt X(f)}^\Star$-modules. The following result is obtained in a standard way, and can be used for the definition of $\pDR_{\wt X(f)}^\Star\cM$ for $\cM$ in $\catD^\rb(\cD_X)$.

\begin{lemme}\label{lem:41}
For a left $\cD_X$-module $\cM$ we have
\tagdroite
\begin{align*}
\pDR_{\wt X(f)}^\Star\cM&\simeq\cHom_{\cD_{\wt X(f)}^\Star}\bigl(\Sp(\cD_{\wt X(f)}^\Star),\varpif^\Star\cM\bigr)
\\
&\simeq\bR\cHom_{\cD_{\wt X(f)}^\Star}\bigl(\cA_{\wt X(f)}^\Star,\varpif^\Star\cM\bigr)
\\
&\simeq\bR\cHom_{\varpif^{-1}\cD_X}\bigl(\varpif^{-1}\cO_X,\varpif^\Star\cM\bigr).\tag*{\qed}
\end{align*}
\taggauche
\end{lemme}

\begin{remarque}\label{rem:gammaf}
The good behavior of $\pDR_{\wt X(f)}^\Star\cM$ with respect to projective pushforwards is recalled in Corollary~\ref{cor:473}.
In particular, recall the graph inclusions $\gamma_f:X\hto X\times\CC$ and $\wt\gamma_f:\wt X(f)\hto X\times\wt\CC$. Then $\bR\wt\gamma_{f*}\pDR^\Star_{\wt X(f)}(\cM)\simeq\pDR^\Star_{X\times\wt\CC}(\Dm\gamma_{f*}\cM)$.
\end{remarque}

The de~Rham complexes with $\Star\in\{\ssup\rmod,\ssup\rrd,\rmodrd\}$ are supported on $\partial\wt X(f)$. We~have the following natural distinguished triangles
\begin{equation}\label{eq:DRStar}
\begin{gathered}
\pDR_{\wt X(f)}^\rmod\cM\to\pDR_{\wt X(f)}^*\cM\to\pDR_{\wt X(f)}^\srmod\cM\To{+1}\\
\pDR_{\wt X(f)}^\rrd\cM\to\pDR_{\wt X(f)}^*\cM\to\pDR_{\wt X(f)}^\srrd\cM\To{+1}
\end{gathered}
\end{equation}
and
\begin{equation}\label{eq:DRmodrd}
\pDR_{\wt X(f)}^\rrd\cM\to\pDR_{\wt X(f)}^\rmod\cM\to\pDR_{\wt X(f)}^\rmodrd\cM\To{+1}.
\end{equation}

\begin{proposition}\label{prop:*jf*}
If $\cM$ is an object of $\catD^\rb_\coh(\cD_X)$, then the natural morphism
\[
\pDR_{\wt X(f)}^*\cM\to\bR\wtjf{}_*\jf^{-1}\pDR_X\cM
\]
is an isomorphism.
\end{proposition}

\begin{proof}
It is enough to prove the result for a coherent $\cD_X$-module, and the assertion is local. We~will work with the unshifted de~Rham complex $\DR$. We~first note, in~a way similar to \eqref{subsec:sheavesblup0}, that for each $\cO_X$-coherent sheaf~$\cF$, we~have $\bR\wtjf{}_*j_f^{-1}\cF=\wtjf{}_*j_f^{-1}\cF$. (Note however that, since $\wtjf$ is not proper, such an isomorphism does not a priori hold for inductive limits of coherent $\cO_X$-modules.) Then, by using a local resolution of $\cF$ by free $\cO_X$-modules of finite rank, one checks that the natural morphism
\[
\wtjf{}_*\cO_{X^*}\otimes_{\varpif^{-1}\cO_X}\varpif^{-1}\cF\to\wtjf{}_*j_f^{-1}\cF
\]
is an isomorphism. In~order to obtain a similar isomorphism for the de~Rham complex of a coherent $\cD_X$-module $\cM$, we~will make use of a coherent filtration. Let us then choose a local coherent filtration $F_\bbullet\cM$ and filter the de~Rham complex by
\[
F_p\DR_X\cM:=\{F_p\cM\ra\Omega^1_X\otimes F_{p+1}\cM\ra\cdots\}.
\]
It is known that, for $p\gg0$ locally, the natural morphism $F_p\DR_X\cM\to\DR_X\cM$ is a quasi-isomorphism: this is proved in \cite[p.\,55]{malgrange85} for the Spencer complex of a right $\cD_X$-module equipped with a coherent filtration, and the result follows by side-changing. In~particular, the graded complex $\gr^F_p\DR_X\cM$ is quasi-isomorphic to zero locally for $p\gg0$.

We also consider the following filtered complexes:
\begin{align*}
F_p\DR_{\wt X(f)}(\wtjf{}_*\cO_{X^*}\otimes\varpif^{-1}\cM)&:=\{\wtjf{}_*\cO_{X^*}\otimes\varpif^{-1}F_p\cM\\
&\hspace*{2cm}\ra \wtjf{}_*\cO_{X^*}\otimes\varpif^{-1}(\Omega^1_X\otimes F_{p+1}\cM)\ra\cdots\},\\
F_p\DR_{\wt X(f)}(\wtjf{}_*\jf{}^{-1}\cM)&:=\{\wtjf{}_*\jf{}^{-1}F_p\cM\ra\varpif^{-1}\Omega^1_X\otimes \wtjf{}_*\jf{}^{-1}F_{p+1}\cM\ra\cdots\}\\
&{}\phantom{:}\simeq\{\wtjf{}_*\jf{}^{-1}F_p\cM\ra\wtjf{}_*\jf{}^{-1}(\Omega^1_X\otimes F_{p+1}\cM)\ra\cdots\}\\
&{}\phantom{:}=:\wtjf{}_*\jf{}^{-1}F_p\DR_X\cM,
\end{align*}
and $\bR\wtjf{}_*\jf{}^{-1}F_p\DR_X\cM$. Since each term of the complex $\jf{}^{-1}F_p\DR_X\cM$ is $\wtjf{}_*$\nobreakdash-acy\-clic (and since the category of sheaves has enough injectives), the natural morphism $\wtjf{}_*\jf{}^{-1}F_p\DR_X\cM\to\bR\wtjf{}_*\jf{}^{-1}F_p\DR_X\cM$ is an isomorphism. We~thus obtain a commutative diagram:
\[
\xymatrix@R=.5cm{
F_p\DR_{\wt X(f)}(\wtjf{}_*\cO_{X^*}\otimes\varpif^{-1}\cM)\ar[r]\ar[d]_\wr&\DR_{\wt X(f)}(\wtjf{}_*\cO_{X^*}\otimes\varpif^{-1}\cM)\ar[ddd]\\
F_p\DR_{\wt X(f)}(\wtjf{}_*\jf{}^{-1}\cM)\ar[d]_\wr\\
\wtjf{}_*\jf{}^{-1}F_p\DR_X\cM\ar[d]_\wr\\
\bR\wtjf{}_*\jf{}^{-1}F_p\DR_X\cM\ar[r]^-\sim_-{p\gg0}&\bR\wtjf{}_*\jf{}^{-1}\DR_X\cM
}
\]
and the right vertical morphism is an isomorphism if and only if the upper horizontal one is so. For the latter, it~is enough to show that for $p\gg0$,
\[
\gr^F_p\DR_{\wt X(f)}(\wtjf{}_*\cO_{X^*}\otimes\varpif^{-1}\cM)\simeq0.
\]
By the left vertical isomorphisms in the above diagram, this complex is identified with $\bR\wtjf{}_*\jf{}^{-1}\gr^F_p\DR_X\cM$, hence is zero locally for $p\gg0$, since $\gr^F_p\DR_X\cM$ is so.
\end{proof}

We note that
\[
\wtjf{}^{-1}\pDR_{\wt X(f)}^\rrd\cM=\wtjf{}^{-1}\pDR_{\wt X(f)}^\rmod\cM=\wtjf{}^{-1}\pDR_{\wt X(f)}^*\cM=\wtjf{}^{-1}\pDR_{\wt X(f)}\cM.
\]
As a consequence, for $\cM$ in $\catD^\rb_\coh(\cD_X)$, we~can identify (possibly non-functorially, due to non-uniqueness in the cone construction) the natural distinguished triangles \eqref{eq:DRStar} with the distinguished triangles
\begin{equation}\label{eq:DRStarbis}
\begin{gathered}
\pDR_{\wt X(f)}^\rmod\cM\to\bR\wtjf{}_*\wtjf{}^{-1}\pDR_{\wt X(f)}^\rmod\cM\to\bR\wtif{}_*\wtif{}^!\pDR_{\wt X(f)}^\rmod\cM[1]\To{+1}\\
\pDR_{\wt X(f)}^\rrd\cM\to\bR\wtjf{}_*\wtjf{}^{-1}\pDR_{\wt X(f)}^\rrd\cM\to\bR\wtif{}_*\wtif{}^!\pDR_{\wt X(f)}^\rrd\cM[1]\To{+1}.
\end{gathered}
\end{equation}

Note however that there exist functorial isomorphisms in $\catD^\rb(\CC_{\wt X(f)})$:
\begin{equation}\label{eq:>rdrd}
\begin{aligned}
\pDR_{\wt X(f)}^\srmod\cM&\isom\bR\wtif{}_*\wtif{}^!\pDR_{\wt X(f)}^\rmod\cM[1],\\
\pDR_{\wt X(f)}^\srrd\cM&\isom\bR\wtif{}_*\wtif{}^!\pDR_{\wt X(f)}^\rrd\cM[1].
\end{aligned}
\end{equation}
Indeed, since $\pDR_{\wt X(f)}^\srmod\cM$ is supported on $\partial\wt X(f)$, the natural morphism
\[
\bR\wtif{}_*\wtif{}^!\pDR_{\wt X(f)}^\srmod\cM\to\pDR_{\wt X(f)}^\srmod\cM
\]
is an isomorphism. Furthermore, since $\bR\wtif{}_*\wtif{}^!\pDR_{\wt X(f)}^*\cM\simeq0$ according to Proposition~\ref{prop:*jf*}, the morphism induced by that in the first line of \eqref{eq:DRStar}:
\[
\bR\wtif{}_*\wtif{}^!\pDR_{\wt X(f)}^\srmod\cM\to\bR\wtif{}_*\wtif{}^!\pDR_{\wt X(f)}^\rmod\cM[1]
\]
is also an isomorphism. We~argue similarly with $\pDR_{\wt X(f)}^\srrd\cM$.\qed

\smallskip
From Properties \eqref{subsec:sheavesblup5} and \eqref{subsec:sheavesblup6}, and due to the flatness property \eqref{subsec:sheavesblup1}, we~obtain the following lemma.

\begin{lemme}\label{lem:48}
For a $\cD_X$-module $\cM$ we have
\begin{align*}
\bR\varpif{}_*\pDR_{\wt X(f)}^*\cM&\simeq\bR\jf{}_*\jf^{-1}\pDR_X\cM,
\\
\bR\varpif{}_*\pDR_{\wt X(f)}^\rmod\cM&\simeq\pDR_X(\cM(*X_0)),
\\
\bR\varpif{}_*\pDR_{\wt X(f)}^\rrd\cM&\simeq\textup{Cone}\bigl[\pDR_X\cM\ra\pDR(\cO_{\wh{X|X_0}}\otimes_{\cO_X}\cM)\bigr][-1]
\\
&\hspace*{-.5cm} =\textup{Cone}\bigl[\pDR_X(\cM(*X_0))\ra\pDR(\cO_{\wh{X|X_0}}\otimes_{\cO_X}\cM(*X_0))\bigr][-1],
\\
\bR\varpif{}_*\pDR_{\wt X(f)}^\rmodrd\cM&\simeq\pDR(\cO_{\wh{X|X_0}}\otimes_{\cO_X}\cM(*X_0)),\\
\bR\varpif{}_*\pDR_{\wt X(f)}^\srmod\cM&\simeq\bR\iif{}_*\iif{}^!\pDR_X(\cM(*X_0))[1].
\end{align*}
\end{lemme}

\begin{proof}
By the projection formula \cite[Prop.\,2.6.6]{K-S90} together with flatness \ref{subsec:sheavesblup1}, we~obtain
\[
\bR\varpi_{f*}\bigl(\cA_{\wt X(f)}^\Star\otimes_{\varpi_f^{-1}\cO_X}\varpi_f^{-1}\cM\bigr)=\bR\varpi_{f*}\bigl(\cA_{\wt X(f)}^\Star\otimes^{\bL}_{\varpi_f^{-1}\cO_X}\varpi_f^{-1}\cM\bigr)\simeq \bR\varpi_{f*}\cA_{\wt X(f)}^\Star\otimes^{\bL}_{\cO_X}\cM,
\]
and the desired formulas follow from the formulas already established for $\bR\varpi_{f*}\cA_{\wt X(f)}^\Star$.
\end{proof}

\begin{theoreme}\label{th:DRmodRconst}
Let $\cM$ be a holonomic $\cD_X$-module. Then $\pDR_{\wt X(f)}^\Star\cM$ belongs to $\catD^\rb_{\RR\textup{-c}}(\CC_{\wt X(f)})$.
\end{theoreme}

\begin{proof}
The case $\Star=*$ will follow from Theorem~\ref{th:q0-1} and we postpone its proof. It~is then enough to prove the cases $\Star=\rmod$ and $\Star=\rrd$. Recall that, according to the microlocal characterization of constructibility of \cite[Th.\,8.4.2]{K-S90}, the notion of $\RR$-constructibility is local. By~a standard `dévissage' (possible due to Corollary~\ref{cor:473}), we~can assume that $\cM$ is a meromorphic flat bundle on~$X$ with pole divisor $P$ containing~$X_0$. As we can work locally on $X$, we~can find a projective modification such that the pole divisor of the pull-back connection has only normal crossings with smooth components. Moreover, according to the theorem of Kedlaya and Mochizuki (\cf\cite{Kedlaya10}, and \cite{Mochizuki09} in the algebraic case), up to blowing-up more, we~can also assume that the pull-back connection has a \emph{good formal structure} along its normal crossing pole set~$D$. We~denote by $e:(Y,D)\to(X,P)$ the projective modification thus obtained and we set $g=f\circ e$, $D_g=g^{-1}(0)=e^{-1}(X_0)\subset D$ (it is the union of some components of $D$). We~consider the commutative diagram
\begin{equation}\label{eq:blowtilde}
\begin{array}{c}
\xymatrix@C=1.2cm{
\wt Y(D)\ar[r]^-{\varpi_{D,D_g}}\ar@/_1pc/[rr]_(.6){\varpi_{D,g}}
\ar@<1ex>@/^2pc/[rrr]^-{\wt\epsilon}\ar@/_.8pc/[drr]_{\varpi_D}
&\wt Y(D_g)\ar[r]^-{\varpi_{D_g,g}}&\wt Y(g)\ar[r]^-{\wt e}\ar[d]_{\varpi_g}&\wt X(f)\ar[d]^{\varpi_f}\\
&&Y\ar[r]^-e&X
}
\end{array}
\end{equation}
The spaces $\wt Y(D)$ and $\wt Y(D_g)$, which are defined in Section~\ref{subsec:realblowupD}, are complex manifolds with corners. We~set $D=D_g\cup D'$, where $D'$ has no common component with $D_g$. It~will be useful to distinguish between the behaviors along $D_g$ and $D'$.

Let $\cM':=\Dm e^*\cM$ be the pull-back meromorphic flat bundle on $Y$. Then one can check that $\cM$ is recovered from $\cM'$ by pushforward (in the sense of $\cD$-modules): $\cM=\Dm e_*\cM'$.

\begin{remarque}\label{rem:Nilssonmod}
In \cite{Mochizuki10}, the de~Rham functors $\pDR_{\wt Y(D)}^{\leq D}, \pDR_{\wt Y(D)}^{<D_g,\leq D'}$ etc. are considered, where the symbol $\leq D$ refers to coefficients in the Nilsson class, and $<D$ to rapid decay coefficients. On the other hand, we~will consider the de~Rham functors $\pDR_{\wt Y(D)}^{\rmod D}, \pDR_{\wt Y(D)}^{\rrd D_g,\rmod D'}$ etc., with the more general condition of moderate growth. This is not problematic since we only consider these functors when applied to meromorphic flat bundles that have a \emph{good formal structure} along $(Y,D)$. In~this case, the natural morphism between the corresponding de~Rham functors is an isomorphism: the identification of the $\cH^0$s is easy in such a case, so that the assertion follows from \cite[Prop.\,5.1.3]{Mochizuki10} and Theorem~\ref{th:drmod} below.
\end{remarque}

We then have:\enlargethispage{\baselineskip}
\begin{equation}\label{eq:Dg}
\begin{split}
\pDR_{\wt Y(g)}^{\rmod}\cM'&\simeq\bR\varpi_{D,g*}\pDR_{\wt Y(D)}^{\rmod D}\cM',\\
\pDR_{\wt Y(g)}^{\rrd D_g}\cM'&\simeq\bR\varpi_{D,g*}\pDR_{\wt Y(D)}^{\rrd D_g,\rmod D'}\cM'.
\end{split}
\end{equation}

\begin{proof}\mbox{}
\begin{itemize}
\item
For the first line, by Lemma~\ref{lem:41} and flatness of $\cA_{\wt Y(D)}^\rmod$ over $\varpi_D^{-1}\cO_X$, we find
\begin{align*}
\bR\varpi_{D,g*}\pDR_{\wt Y(D)}^{\rmod D}\cM'&\simeq \bR\varpi_{D,g*}\bR\cHom_{\varpi_D^{-1}\cD_X}\bigl(\varpi_D^{-1}\cO_X,\varpi_D^\rmod\cM'\bigr)\\
&\simeq\bR\cHom_{\varpi_g^{-1}\cD_X}\bigl(\varpi_g^{-1}\cO_X,\bR\varpi_{D,g*}\varpi_D^\rmod\cM'\bigr)\;\;\text{\cite[(2.6.15)]{K-S90}}
\end{align*}
and by the same argument as in the proof of Lemma \ref{lem:48}, we have
\[
\bR\varpi_{D,g*}\varpi_D^\rmod\cM'\simeq\varpig^\rmod\cM'.
\]

\item
The second line is obtained by applying results of \cite{Mochizuki10} for Nilsson class functions, as we have not proved $\varpi_D^{-1}\cO_Y$-flatness of $\cA_{\wt Y(D)}^{\rrd D_g,\rmod D'}$, and we use the remark above on $\wt Y(D)$. We thus apply first \cite[Lem.\,5.1.6]{Mochizuki10} to the morphism $\varpi_{D,D_g}$, and then \cite[Prop.\,4.7.4]{Mochizuki10} to $\rho=\varpi_{D_g,g}$.\qedhere
\end{itemize}
\end{proof}

Since $\RR$-constructibility is stable under proper pushforward by a real analytic map between real analytic manifolds, we~conclude from \eqref{eq:Dg} and Corollary~\ref{cor:473} (due to the remarks in Section~\ref{sec:realblowup}), that it is enough to prove $\RR$-constructibility of $\pDR_{\wt Y(D)}^{\rmod D}\cM'$ and $\pDR_{\wt Y(D)}^{\rrd D_g,\rmod D'}\cM'$. The generalized Hukuhara-Turrittin theorem yields the following consequence (\cf \cite[\S12.d]{Bibi10}):

\begin{theoreme}\label{th:drmod}
Let $\cM'$ be a meromorphic flat bundle on $Y$ with poles along $D$. Assume moreover that $\cM'$ has a good formal structure along $D$. Then the complexes $\DR_{\wt Y(D)}^{\rrd D}\cM'$, $\DR_{\wt Y(D)}^{\rrd D_g,\rmod D'}\cM'$ and $\DR_{\wt Y(D)}^{\rmod D}\cM'$ have cohomology in degree zero at most, and their $\cH^0$ are nested subsheaves of the local system $\cH^0\DR_{\wt Y(D)}^*\cM'$, which are $\RR$-constructible with respect to any stratification on $\partial\wt Y(D)$ determined by the good stratified $\ccI$-covering associated with $\cM'$ (\cf Section~\ref{subsec:Icovering}).\qed
\end{theoreme}

This concludes the proof of Theorem~\ref{th:DRmodRconst} (modulo the case $\Star=*$).
\end{proof}

\begin{exemple}[Regular singularities]
Assume that $\cM$ is a regular holonomic $\cD_X$-module. Then, in the `dévissage' aforementioned, the meromorphic bundle with flat connection $\cM'$ has regular singularities along $D$. In~Theorem~\ref{th:drmod}, one finds that the sheaf $\cH^0\DR_{\wt Y(D)}^{\rrd D_g,\rmod D'}\cM'$ vanishes along $\varpi_D^{-1}(D_g)$ and is equal to the local system $\cH^0\DR_{\wt Y(D)}^*\cM'$ when restricted to $\varpi_D^{-1}(D\moins D_g)$. Similarly, $\cH^0\DR_{\wt Y(D)}^{\rmod D}\cM'$ is equal to $\cH^0\DR_{\wt Y(D)}^*\cM'$. One concludes that
\begin{align*}
\pDR_{\wt X(f)}^{\rmod}\cM&\simeq\pDR_{\wt X(f)}^*\cM=\bR\wtj_{f*}\,\wtj_f^{-1}\pDR_X\cM,\\
\pDR_{\wt X(f)}^{\rrd}\cM&\simeq\bR\wtj_{f!}\,\wtj_f^{-1}\pDR_X\cM,\\
\pDR_{\wt X(f)}^{\ssup\rmod}\cM&=0.
\end{align*}
\end{exemple}

\subsection{Duality properties}\label{subsec:dual}
In a first version of this article, the next theorem was a conjecture, which has been successfully proved by B.\,Hepler and A.\,Hohl in the beautiful paper \cite{H-H25}.

\begin{theoreme}[Behaviour with respect to duality, {\cite[Th.\,7.4]{H-H25}}]\label{conj:DRdual}
Let $\cM$ be a holonomic $\cD_X$-module.
\begin{enumerate}
\item\label{conj:DRdual1}
We have a functorial isomorphism
\[
\bD\pDR_{\wt X(f)}^\rmod\cM\simeq\pDR_{\wt X(f)}^\rrd\bD\cM,
\]
so that the dual of the distinguished triangle
\[
\pDR_{\wt X(f)}^\rmod\cM\to\pDR_{\wt X(f)}^*\cM\to\pDR_{\wt X(f)}^\srmod\cM\To{+1}
\]
is isomorphic to the natural triangle
\[
\bR\wtif{}_*\wtif{}^{-1}\pDR_{\wt X(f)}^\rrd\bD\cM[-1]\to\bR\wt j_{f!}\,\jf^{-1}\pDR_X\bD\cM\to\pDR_{\wt X(f)}^\rrd\bD\cM\To{+1}.
\]

\item\label{conj:DRdual2}
We have a functorial isomorphism
\[
\bD\pDR_{\wt X(f)}^\rrd\cM\simeq\pDR_{\wt X(f)}^\rmod\bD\cM,
\]
so that the dual of the distinguished triangle
\[
\pDR_{\wt X(f)}^\rrd\cM\to\pDR_{\wt X(f)}^*\cM\to\pDR_{\wt X(f)}^\srrd\cM\To{+1}
\]
is isomorphic to the natural triangle
\[
\bR\wtif{}_*\wtif{}^{-1}\pDR_{\wt X(f)}^\rmod\bD\cM[-1]\to\bR\wt j_{f!}\,\jf^{-1}\pDR_X\bD\cM\to\pDR_{\wt X(f)}^\rmod\bD\cM\To{+1}.
\]
\end{enumerate}
\end{theoreme}

In Appendix~\ref{app:C}, we~give a proof of Conjecture~\ref{conj:DRdual} in a special case, namely, in a local setting and when~$\cM$ is a meromorphic flat bundle on $X$, by applying results of \cite{Mochizuki10} (\cf Proposition~\ref{prop:duality}). However, functoriality of this isomorphism is lacking in order to obtain the general case, which has been obtained in \cite{H-H25}. As~the statement of \loccit\ only concerns the first line in \eqref{conj:DRdual1} and \eqref{conj:DRdual2}, let us make clear how to obtain the identifications of the distinguished triangles. We~treat the first case for example.

The natural morphism $\pDR_{\wt X(f)}^\rmod\cM\to\pDR_{\wt X(f)}^*$ is identified with the natural morphism $\pDR_{\wt X(f)}^\rmod\cM\to\bR\wtj{}_*\wtj^{-1}\pDR_{\wt X(f)}^\rmod\cM$, according to the identification of \eqref{eq:DRStar} with~\eqref{eq:DRStarbis}; thus,~the behavior of Poincaré-Verdier duality with respect to complementary open and closed inclusions yields the result. At this point, we~do not obtain functoriality, due to the cone construction in the latter identification.

\begin{corollaire}
Let $\cM$ be a holonomic $\cD_X$-module. We~have an isomorphism
\[
\bD\pDR_{\wt X(f)}^\rmodrd\cM\simeq(\pDR_{\wt X(f)}^\rmodrd\bD\cM)[-1].\eqno\qed	
\]
\end{corollaire}

\subsection{Non-characteristic properties}\label{subsec:noncar}
Let $p:X\to \Delta$ be a smooth holomorphic map to a disc $\Delta$ with coordinate $t$. We set $\Deltastar=\Delta\moins\{0\}$. For $s\in \Delta$, we~denote by $f_s:X_s\to\CC$ the function induced by $f$ on $X_s:=p^{-1}(s)$, and by $i_s:X_s\hto X$ and $\wti_s:\wt X_s(f_s)\hto\wt X(f)$ the inclusions.

\begin{proposition}\label{prop:noncharA}
For each $x_o\in p^{-1}(0)$ there exists an open neighborhood of $x_o$ in $X$ such that, up to shrinking $\Delta$, the following holds for each $s\in \Delta\moins\{0\}$:
\begin{enumerate}
\item\label{prop:noncharA1}
$\wt X_s(f_s)=\wt X(f)_{|X_s}$,
\item\label{prop:noncharA2}
$\cA^\Star_{\wt X_s(f_s)}=\wti{}_s^*\cA^\Star_{\wt X(f)}=\bL\wti{}_s^*\cA^\Star_{\wt X(f)}$.
\end{enumerate}
\end{proposition}

\begin{proof}
We consider the setting and notation of the proof of Theorem~\ref{th:DRmodRconst}, and $X$ still denotes a small neighborhood of $x_o$. We~can thus assume that $p$ is a projection of a product $X_{s=0}\times\Delta\to\Delta$. Since $e$ is proper, we~can shrink $X$ and $\Delta$ so that, on each stratum of the natural stratification of $D$ that lies over $\Deltastar$, $p\circ e$ has rank one, \ie is a submersion from such a stratum to $\Deltastar$. Locally near a point of $Y\moins(p\circ e)^{-1}(0)$, we~can find local coordinates $(y_1,\dots,y_n)$ such that $D=\{y_1\cdots y_\ell=0\}$ ($\ell<n$) and $p\circ e=y_n$. Let us check that the assertions corresponding to \eqref{prop:noncharA1} and \eqref{prop:noncharA2} for $\wt Y(D)$ and $p\circ e$ are true in this local setting, hence all over $\wt Y(D)$. This is clear for \eqref{prop:noncharA1}. For \eqref{prop:noncharA2}, this is clear for $\cA^*_{\wt Y(D)}=\wtj_*\cO_{Y^*}$. We~can argue with the maximum principle as in \cite[Lem.\,4.4.1]{Mochizuki10} for $\cA_{\wt Y(D)}^\rmod$ and $\cA_{\wt Y(D)}^\rrd$. It~remains to show the injectivity of $t-s$ on $\wtj_*\cO_{Y^*}/\cA_{\wt Y(D)}^\rmod$, $\wtj_*\cO_{Y^*}/\cA_{\wt Y(D)}^\rrd$ and $\cA_{\wt Y(D)}^\rmod/\cA_{\wt Y(D)}^\rrd$. This is obtained by the same argument using the maximum principle.

By computation in these local coordinates, we note that, for $s\neq0$, $\wt Y_s(D_s)$ is the closure of $Y_s^*$ in $\wt Y(D)$, and we have a similar property for $\wt Y_s(g_s)$. Since $\varpi_{D,g}$ is proper, we~conclude that \eqref{prop:noncharA1} holds for $\wt Y(g)$. Using now the properness of $\wt e$, we~obtain similarly \eqref{prop:noncharA1} for $\wt X(f)$.

Now, \eqref{prop:noncharA2} for $\wt X(f)$ is obtained from \eqref{prop:noncharA2} for $\wt Y(D)$ by using \ref{subsec:sheavesblupD2} and \ref{eq:pusheA}.
\end{proof}

We consider the sheaf $\cD_{X/\Delta}$ of relative differential operators, which is a subsheaf of~$\cD_X$ and, for a holonomic $\cD_X$-module $\cM$, the relative de~Rham complex $\DR_{X/\Delta}\cM$, which is a complex of $p^{-1}\cO_\Delta$-modules. By~pulling it back to $\wt X(f)$ and tensoring the terms with $\cA_{\wt X(f)}^\Star$, we~obtain the relative $\Star$de~Rham complex $\DR^\Star_{\wt X(f)/\Delta}\cM$, which is a complex of $(p\circ\pi_f)^{-1}\cO_\Delta$-modules.

On the other hand, since $p$ is assumed to be a projection, $\DR^\Star_{\wt X(f)}\cM$ is the single complex associated with the double complex
\[
\DR^\Star_{\wt X(f)/\Delta}\cM\To{\partial_t}\DR^\Star_{\wt X(f)/\Delta}\cM,
\]
and the natural (vertical) morphism of complexes
\[
\xymatrix{
\DR^\Star_{\wt X(f)/\Delta}\cM\ar[r]^-{\partial_t}\ar@{=}[d]&\DR^\Star_{\wt X(f)/\Delta}\cM\ar[d]\\
\DR^\Star_{\wt X(f)/\Delta}\cM\ar[r]&0
}
\]
induces a $(p\circ\varpif)^{-1}\cO_\Delta$-linear morphism $(p\circ\varpif)^{-1}\cO_\Delta\otimes_\CC\DR^\Star_{\wt X(f)}\cM\to\DR^\Star_{\wt X(f)/\Delta}\cM$.

\begin{proposition}\label{prop:XsurS}
Let $\cM$ be a holonomic $\cD_X$-module. Up to shrinking $X$ (and restricting away from $p^{-1}(0)$), the natural morphism
\[
(p\circ\nobreak\varpif)^{-1}\cO_{\Deltastar}\otimes_\CC\DR^\Star_{\wt X(f)}\cM|_{\Deltastar}\to\DR^\Star_{\wt X(f)/\Delta}\cM|_{\Deltastar}
\]
is a quasi-isomorphism.
\end{proposition}

\begin{corollaire}\label{cor:XsurS}
With the assumptions above, for $s\neq0$ we have
\[
\cH^k(\Dm i^*_s\cM)=0\quad \text{if $k\neq0$},
\]
and a functorial isomorphism
\[
\wti_s^{-1}\DR_{\wt X(f)}^\Star\cM\simeq \DR_{\wt X_s(f_s)}^\Star \Dm i^*_s\cM.
\]
\end{corollaire}

\begin{proof}
For $X$ small enough, $X_s$ is non-characteristic for $\cM$ if $s\neq0$, hence the first point. Then, according to Proposition~\ref{prop:noncharA},
\[
\DR_{\wt X_s(f_s)}^\Star \Dm i^*_s\cM\simeq \bL \wti_s^*\DR^\Star_{\wt X(f)/\Delta}\cM,\quad\text{if }s\neq0.
\]
We then conclude the proof by using Proposition~\ref{prop:XsurS}.
\end{proof}

\begin{remarque}\label{rem:XsurS}
By definition, $\pDR_{\wt X(f)}^\Star\cM$ has nonzero cohomology in non-positive degrees at most. On the other hand, we~claim that $\cH^0\pDR_{\wt X(f)}^\Star\cM=0$ away from the pull-back by $\varpif$ of a discrete set of points in $X$. Indeed, let $x_o\in X$ and let $p:\nb(x_o)\to \Delta$ be a smooth function defined in a neighborhood of $x_o$, which we still denote by $X$. Then Corollary~\ref{cor:XsurS} reads (in terms of the shifted de~Rham complex)
\[
i^{-1}_s\pDR_{\wt X(f)}^\Star\cM\simeq \pDR_{\wt X_s(f_s)}^\Star\cH^0(\Dm i^*_s\cM)[1]\quad\text{for $s\neq0$},
\]
hence the vanishing of $i^{-1}_s\cH^0\pDR_{\wt X(f)}^\Star\cM$ for $s\neq0$. One obtains the assertion by applying this to the projections along all coordinate hyperplanes centered at $x_o$.
\end{remarque}

\begin{proof}[\proofname\ of Proposition~\ref{prop:XsurS}]
Since the argument for proving Corollary~\ref{cor:473} relies on \cite[Th.\,4.1.5]{Mochizuki10}, one obtains that it holds for the relative $\Star$\,de~Rham complex, provided $p\circ\pi$ is smooth. Similarly, \eqref{eq:Dg} holds in the relative case provided $p\circ e$ is smooth. The smoothness assumption holds when we restrict to $\Deltastar$ possibly shrunk, if $(X,x_o)$ is small enough.

We take up the setting and notation of the proof of Theorem~\ref{th:DRmodRconst}, in particular as indicated in the diagram \eqref{eq:blowtilde}. Then, according to the preliminary remark above, we~have (away from $p^{-1}(0)$ for the second line):
\begin{align*}
\bR\wt \epsilon_*\DR_{\wt Y(D)}^\Star\cM'&\simeq \DR_{\wt X(f)}^\Star \Dm e_*\cM',\\
\bR\wt \epsilon_*\DR_{\wt Y(D)/\Delta}^\Star\cM'|_{\Deltastar}&\simeq \DR_{\wt X(f)/\Delta}^\Star \Dm e_*\cM'|_{\Deltastar}.
\end{align*}
On the other hand,
\begin{align*}
\bR\wt \epsilon_*\Bigl[(p\circ e\circ\varpi_D)^{-1}\cO_\Delta\otimes_\CC\DR_{\wt Y(D)}^\Star\cM'\Bigr]&\simeq\bR\wt \epsilon_*\Bigl[(p\circ\varpif\circ\wt \epsilon)^{-1}\cO_\Delta\otimes_\CC\DR_{\wt Y(D)}^\Star\cM'\Bigr]\\
&\simeq(p\circ\varpif)^{-1}\cO_\Delta\otimes_\CC\bR\wt \epsilon_*\DR_{\wt Y(D)}^\Star\cM'\\
&\simeq(p\circ\varpif)^{-1}\cO_\Delta\otimes_\CC\DR_{\wt X(f)}^\Star \Dm e_*\cM',
\end{align*}
so that it is enough to prove the proposition for $\wt Y(D)$ and $p\circ e$ (away from $(p\circ e)^{-1}(0)$). The case when $\Star=*$ being easy, we~are reduced to checking the cases when $\Star=\rrd$ and $\Star=\rmod$. We~can now work locally on $(Y,D)$ near a point $y_o\in e^{-1}(x_o)$, due to the properness of $e$. We~then use the notation $\cM$ instead of $\cM'$. We~choose local coordinates near a point of a neighborhood of $y_o$ not in $(p\circ e)^{-1}(0)$, as in the proof of Proposition~\ref{prop:noncharA}.

Let $\rho_{\bmd}$ be a local ramification along the components of $D$. Then $\cM$ is a direct summand of $\Dm\rho_{\bmd*}\Dm\rho_{\bmd}^*\cM$ so, by the pushforward argument already used, we~can assume that $\cM$ has a good formal decomposition along $D$. According to the generalized Hukuhara-Turrittin theorem already used in Theorem~\ref{th:drmod} (\cf \eg\cite[Th.\,12.5]{Bibi10} and the references given therein), we~can reduce to the case where $\cM=\cE^\varphi\otimes\cR$, where $\cE^\varphi=(\cO_Y(*D),\rd+\rd\varphi)$ and $\varphi$ is purely monomial, and $\cR$ has a regular singularity. By~induction on the rank of $\cR$, we~can assume that $\cR$ has rank one as an $\cO_Y(*D)$-module.

By the theorems of Majima \cite{Majima84} (\cf also \cite[App.]{Bibi93} for the rapid-decay case and \cite[App.]{Hien07} for the case with moderate growth), one proves that both the relative and the absolute de~Rham complexes (in the variants $\rrd$ and $\rmod$) have cohomology in degree zero only. Due to the special form of $\cM$, computing the $\cH^0$ of these complexes is easy, by twisting with $e^{-\varphi}$, and the desired isomorphism is then straightforward to obtain, as it is clear to decide whether $e^{-\varphi}x^\alpha$ $(\alpha\in\CC^\ell$) has rapid decay (\resp moderate growth) in any given small multi-sector.
\end{proof}

We can define the $X$-support condition as in Definition~\ref{def:X0support}.

\begin{corollaire}\label{cor:X-supp}
For a holonomic $\cD_X$-module $\cM$, the complexes $\pDR^\rmod\cM$ and $\pDR^\rrd\cM$ and their Poincaré-Verdier duals satisfy the $X$-support condition.
\end{corollaire}

\begin{proof}
The proof for the dual complexes follows from that for the complexes themselves, according to the Hepler--Hohl theorem \ref{conj:DRdual}. We~then argue as for the complex $\pDR\cM$. From Remark~\ref{rem:XsurS}, we obtain the inequality of Definition~\ref{def:X0support} for $j=0$. We argue by induction for $j\geq1$. The question is local and we can work in the setting of Corollary~\ref{cor:XsurS}. By induction, we have
\[
\dim X_s\text{-support}\, \bigl[\cH^j\pDR_{\wt X_s(f_s)}^\Star\cH^0(\Dm i^*_s\cM)\bigr]\leq-j,
\]
and so, by Corollary~\ref{cor:XsurS},
\begin{align*}
\dim X_s\text{-support}\, \bigl[\cH^j\bigl(i^{-1}_s\pDR_{\wt X(f)}^\Star\cM\bigr)\bigr]&=\dim X_s\text{-support}\, \bigl[\cH^{j+1}\pDR_{\wt X_s(f_s)}^\Star\cH^0(\Dm i^*_s\cM)\bigr]\\
&\leq-j-1.
\end{align*}
By considering various projections $p$, we conclude that
\[
\dim X\text{-support}\,\bigl(\cH^j\pDR_{\wt X(f)}^\Star\cM\bigr)\leq-j.\qedhere
\]
\end{proof}

\section{The sheaf of nearby cycles as a sheaf on the real blow-up space}\label{sec:nearby}
\subsection{The functor \texorpdfstring{$\ppsifstar$}{psif}}
Let $\cF$ be an object of $\catD^\rb(\CC_X)$. We~set
\[
\psifstar\cF:=\wtif{}^{-1}\bR\wtjf{}_*\,\wtjf{}^{-1}\cF,
\]
and $\ppsifstar\cF:=\psifstar\cF[-1]$ (recall that, similarly, $\ppsif\cF:=\psif\cF[-1]$).

\begin{remarque}\label{rem:Rconstr}
If $\cF$ is an object of $\catD^\rb_\Rc(\CC_X)$, then $\psifstar\cF$ is an object of $\catD^\rb_\Rc(\CC_{\partial\wt X})$:
\begin{itemize}
\item
the pullback $\varpi_f^{-1}\cF$ is $\RR$-constructible, as follows from \cite[Prop.\,8.4.10(i)]{K-S90};
\item
weak $\RR$-constructibility $\bR\wtjf{}_{_{\scriptstyle!}}\,\wtjf{}^{-1}\cF$ follows from the existence of a subanalytic refinement compatible with the pair $(\wt X,\partial\wt X)$ of a given subanalytic stratification of $\wt X$, and the finiteness property for obtaining $\RR$-constructibility is clear since $\bR\wtjf{}_{_{\scriptstyle!}}\,\wtjf{}^{-1}\cF$ is zero on $\partial\wt X$;
\item
by duality (\cf \cite[Prop.\,8.4.9]{K-S90}), $\bR\wtjf{}_*\,\wtjf{}^{-1}\cF$ is $\RR$-constructible, and applying once more \cite[Prop.\,8.4.10(i)]{K-S90} one obtains the $\RR$-constructibility of $\psifstar\cF$.
\end{itemize}
\end{remarque}

\begin{lemme}\label{lem:pushforwardwtpsifF}
Let $\pi:Y\to X$ be a proper morphism between complex manifolds and set $g=f\circ\pi$. For $\cG$ in $\catD^\rb(\CC_Y)$, we~have a functorial isomorphism
\[
\bR\wt\pi_*\ppsigstar\,\cG\simeq\ppsifstar\bR \wt\pi_*\cF.
\]
\end{lemme}

\begin{proof}
The lemma immediately follows from the base change theorem for a proper morphism and the property that $\wt Y(g)=Y\bigtimes_X\wt X(f)$.
\end{proof}

As a consequence, one can reduce the computation of $\ppsifstar\cF$ to the case where $f$ is the projection $X=X_0\times\CC\to\CC$, by applying the lemma to the graph embedding of~$f$.

For an object $\cM$ of $\catD^\rb(\cD_X)$, we~set for short
\begin{equation}\label{eq:wtppsiM}
\ppsifstar\cM:=\ppsifstar\pDR\cM.
\end{equation}

\begin{lemme}\label{lem:pushforwardwtpsif}
Same setting as in Lemma \ref{lem:pushforwardwtpsifF}. For $\cM$ in $\catD^\rb_{\pi\textup{-good}}(\cD_Y)$ (\ie with good cohomology in the sense of \eg\cite[Def.\,4.24]{Kashiwara03} in a neighborhood of any fiber of~$\pi$), we~have a functorial isomorphism
\[
\bR\wt\pi_*\ppsigstar\cM\simeq\ppsifstar(\Dm\pi_*\cM).\eqno\qed
\]
\end{lemme}

\subsection{Proof of \texorpdfstring{Theorem~\ref{th:q0-1}}{ref}}\label{subsec:pfq0-1}
This theorem follows from Proposition~\ref{prop:psipsitilde} below, which applies to any $\CC$-constructible complex $\cF$. We~will use the notation of the diagram in Figure~\ref{fig:diagram}, where all squares are Cartesian, $\wt X:=\wt X(f)$, and all maps $\rho$ are defined from the universal covering $\RR\to S^1$.
\begin{figure}[htb]
\[
\xymatrix@=.6cm{
&&\RR\times0\ar@{^{ (}->}[rr]\ar[dd]|!{[dl];[dr]}\hole &&\RR\times\RR_+\ar[dd]|!{[dl];[dr]}\hole&&\ar@{_{ (}->}[ll]\RR\times\RR_+^*\ar[dd]\\
\RR\times X_0=\hspace*{-.75cm}&\partial\,\wh X\ar@<-.5ex>@/_1.5pc/[dddd]_(.3){q_0}|(.525)\hole\ar@{^{ (}->}[rr]^(.35){\whi_{\!f}}\ar[dd]^(.3){\rho_0}\ar[ur]&&\wh X\ar[dd]^(.3){\rho}\ar[ur]&&\ar@{_{ (}->}[ll]_(.3){\whj_{\!f}}\wh X{}^*\ar[dd]^(.3){\rho}\ar[ur]\\
&&S^1\times0\ar@{^{ (}->}[rr]|!{[ur];[dr]}\hole\ar[dd]|!{[dl];[dr]}\hole&&S^1\times\RR_+\ar[dd]|!{[dl];[dr]}\hole&&\ar@{_{ (}->}[ll]|!{[ul];[dl]}\hole S^1\times\RR_+^*\ar@{=}[dd]\\
S^1\times X_0=\hspace*{-.75cm}&\partial\wt X\ar@{^{ (}->}[rr]^(.35){\wtif}\ar[dd]^(.3){\varpifo}\ar[ur]^-{\wt f_0}&&\wt X\ar[dd]^(.3){\varpif}\ar[ur]^-{\wt f}&&\ar@{_{ (}->}[ll]_(.3){\wtjf}X^*\ar@{=}[dd]\ar[ur]\\
&&0\ar@{^{ (}->}[rr]|!{[dl];[dr]}\hole&&\CC&&\ar@{_{ (}->}[ll]|!{[ul];[dl]}\hole \CC^*\\
&X_0\ar@{^{ (}->}[rr]^(.35){\iif}\ar[ur]&&X\ar[ur]^-f&&\ar@{_{ (}->}[ll]_(.3){\jf}X^*\ar[ur]^-f
}
\]
\caption{\label{fig:diagram}}
\end{figure}

Let us set $\cF^*=\jf^{-1}\cF$. Our first aim is to express the complex of nearby cycles $(\psif\cF,\rT)$ as defined by Deligne \cite{Deligne73} in terms of $\psifstar\cF$. Let $\sigma_0:\partial\,\wh X\isom\partial\,\wh X$ be the automorphism induced by $\theta\mto\theta+1$ on $\RR$. Since $\rho_0\circ\sigma_0=\rho_0$, we~have an identification
\[
\rho_0^{-1}\psifstar\cF\isom\sigma_0^{-1}\rho_0^{-1}\psifstar\cF,
\]
hence an isomorphism
\[
\wt\rT:\bR\rho_{0*}\rho_0^{-1}\psifstar\cF\to\bR\rho_{0*}\rho_0^{-1}\psifstar\cF,
\]
and thus an automorphism $\wt\rT$ of $\bR q_{0*}\rho_0^{-1}\psifstar\cF$.

\begin{proposition}\label{prop:psipsitilde}
We have a functorial isomorphism $(\psif\cF,\rT)\simeq(\bR q_{0*}\rho_0^{-1}\psifstar\cF,\wt\rT)$ and the morphism
\begin{starequation}\label{eq:psipsitilde}
q_0^{-1}\psif\cF\to \rho_0^{-1}\psifstar\cF
\end{starequation}%
induced by the adjunction $q_0^{-1}\bR q_{0*}\to\id$ is an isomorphism.
\end{proposition}

\begin{lemme}\label{lem:rhoopsitilde}
Let us set $\cF^*:=\jf^{-1}\cF$. We~have a functorial isomorphism
\[
(\bR\rho_{0*}\rho_0^{-1}\psifstar\cF,\wt\rT)\simeq(\wtif{}^{-1}\bR\wtjf{}_*(\bR\rho_*\rho^{-1}\cF^*),\rT).
\]
\end{lemme}

\begin{proof}
We have
\begin{align*}
\wtif{}^{-1}\bR\wtjf{}_*(\bR\rho_*\rho^{-1}\cF^*)&\simeq \wtif{}^{-1}\bR\rho_*\bR\,\,\whj_{\!f}{}_*\rho^{-1}\cF^*\\
&\simeq \wtif{}^{-1}\bR\rho_*\rho^{-1}\bR\wtjf{}_*\cF^*\quad(\text{Example~\ref{ex:basech}})\\
&\simeq \bR\rho_{0*}\,\,\whi_{\!f}^{-1}\rho^{-1}\bR\wtjf{}_*\cF^*\quad(\text{Example~\ref{ex:basech}})\\
&=\bR\rho_{0*}\rho_0^{-1}\wtif{}^{-1}\bR\wtjf{}_*\cF^*=\bR\rho_{0*}\rho_0^{-1}\psifstar\cF.
\end{align*}
The compatibility with $\wt\rT,\rT$ is then clear.
\end{proof}

\begin{proof}[\proofname\ of Proposition \ref{prop:psipsitilde} (first part)]
We have\enlargethispage{\baselineskip}
\begin{align*}
\psif\cF&=\iif^{-1}\bR\jf{}_*\bR\rho_*\rho^{-1}\cF^*\quad\text{(by definition)}\\
&=\iif^{-1}\bR\varpif{}_*\bR\wtjf{}_*\bR\rho_*\rho^{-1}\cF^*\\
&=\bR\varpifo{}_*\wtif{}^{-1}\bR\wtjf{}_*\bR\rho_*\rho^{-1}\cF^*\quad(\varpif\text{ proper})\\
&\simeq\bR\varpifo{}_*\bR\rho_{0*}\rho_0^{-1}\psifstar\cF\quad(\text{Lemma~\ref{lem:rhoopsitilde}})\\
&=\bR q_{0*}\rho_0^{-1}\psifstar\cF.
\end{align*}
The compatibility with $\wt\rT,\rT$ follows from the previous lemma.
\end{proof}

\begin{proof}[\proofname\ that \eqref{eq:psipsitilde} is an isomorphism]

\begin{lemme}\label{lem:ixloccst}
Let $\cG$ be a weakly $\RR$-constructible bounded complex on $\partial\wt X(f)$ (\cf\cite[Def.\,8.4.3]{K-S90}) satisfying the following property:
\par\noindent
\begin{minipage}[t]{.08\textwidth}
$(\ref{lem:ixloccst}\,*)$
\end{minipage}
\hfill
\begin{minipage}[t]{.9\textwidth}
For each $x\in X_0$ and $\wti_x:\varpifo^{-1}(x)\simeq S^1\times\{x\}\hto\partial\wt X(f)\simeq S^1\times X_0$, the cohomology sheaves of the restriction $\wti_x{}^{-1}\cG$ to $S^1\times\{x\}$ are locally constant with finite rank.
\end{minipage}
\par\noindent
Then the adjunction morphism $q_0^{-1}\bR q_{0*}\rho_0^{-1}\cG\to\rho_0^{-1}\cG$ is an isomorphism.
\end{lemme}

\begin{proof}
We first reduce to proving the lemma when $X_0$ is a point. It~is enough to prove that, for every $x\in X_0$, the morphism
\[
\whi_{\!x}{}^{-1}q_0^{-1}\bR q_{0*}\rho_0^{-1}\cG\to\whi_{\!x}{}^{-1}\rho_0^{-1}\cG
\]
is an isomorphism, and this reduces to showing
\[
q_0^{-1}\bR\Gamma(S^1\times\{x\},\wti_{x}^{-1}\bR\rho_{0*}\rho_0^{-1}\cG)\to\rho_0^{-1}\wti_{x}^{-1}\cG
\]
is an isomorphism. Due to the assumption on $\cG$, we~can apply Example~\ref{ex:basech} to write the right-hand side as $q_0^{-1}\bR\Gamma(\RR\times\{x\},\rho_0^{-1}\wti_{x}^{-1}\cG)$, so~we are reduced to proving the lemma for $\wti_{x}^{-1}\cG$ on $S^1\times\{x\}$.

Now, if $\cG$ is a bounded complex on $S^1$ whose cohomology is locally constant and of finite rank, the cohomology of $\rho_0^{-1}\cG$ on $\RR$ is constant of finite rank, and \hbox{$H^k(\RR,\cH^j\cG)\!=\!0$} for $k\neq0$, so~it is easy to conclude that $q_0^{-1}\bR\Gamma(\RR,\rho_0^{-1}\cG)\to\rho_0^{-1}\cG$ is an isomorphism.
\end{proof}

\begin{lemme}
Let $\pi:Y\to X$ be a proper morphism and set $g=f\circ\pi$. If the morphism~\eqref{eq:psipsitilde} for $g$ and a $\CC$-constructible bounded complex $\cG$ on $Y$ is an isomorphism, then so is the morphism~\eqref{eq:psipsitilde} for $f$ and $\cF=\bR\pi_*\cG$ on $X$.
\end{lemme}

\begin{proof}
Straightforward due to the base change property for a proper morphism.
\end{proof}

By a standard `dévissage', we~can assume that there exists a divisor $D'\subset X$ with normal crossings and smooth components, such that, denoting by $j:U=X\moins D'\hto X$ the inclusion, $\cF=j_!\cL$, where $\cL$ is a local system on $U$, and $f^{-1}(0)=D\subset D'$, so~$X_0=D$ with the previous notation. Let $\varpi_D:\wt X(D)\to X$ be the real blowing up of the components of $D$, so~that $\wt X(D)$ is a manifold with corners. Then we have a decomposition $\varpi_D=\varpif\circ\varpi_{D,f}$ with $\varpi_{D,f}:\wt X(D)\to\wt X(f)$. We~will prove that~$(\ref{lem:ixloccst}\,*)$ holds at any $x_o\in D$ for $\psifstar\cF$ with these assumptions.

We can choose local coordinates $(x_1,\dots,x_\ell,y_1,\dots,y_m,z_1,\dots,z_p)$ on $X$ centered at a fixed $x_o\in D$ such that $f(\bmx,\bmy,\bmz)=\prod_{i=1}^\ell x_i^{e_i}=:\bmx^{\bme}$ ($e_i>0$ for all $i$) and $D'=\{\prod x_i\prod y_j=0\}$. Then $\wt X(D)$ has partial polar coordinates $(\rhog,\rme^{i\thetag},\bmy,\bmz)$ ($\rho_i\in\RR_+$) and $\partial\wt X(D)=\{\rho_1\cdots\rho_\ell=0\}$. The map $\varpi_{D,f}:\varpi_D^{-1}(x_o)\to\varpi_f^{-1}(x_o)$ is induced by the map $(S^1)^\ell\to S^1$ given by $\rme^{i\thetag}\mto\rme^{\sum e_i\theta_i}$.

With obvious notation, we~have $\psifstar\cF=\bR\varpi_{D,g*}\,\wti_D^{-1}\bR\wtj_{D*}\,\cF^*$. Restricting to $\prod y_j\neq0$, we~have $\cF^*=\cL$ and $\wt\cL:=\wti_D^{-1}\bR\wtj_{D*}\cF^*$ is a local system with the same monodromy as $\cL$. On $\partial\wt X(D)$ we then have $\wti_D^{-1}\bR\wtj_{D*}\,\cF^*=j_!\wt\cL$ (extension by zero along $\prod y_j=0$).

Since the map $\varpi_{D,f}:\varpi_D^{-1}(x_o)\to\varpi_f^{-1}(x_o)$ is a fibration and $\wt\cL$ is a local system, we~conclude that, if $x_o\in D\moins \{\prod y_j=0\}$, the cohomology sheaves of $\wti_{x_o}^{-1}\psifstar\cF$ are locally constant of finite rank, while if $x_o\in D\cap \{\prod y_j=0\}$, they are zero. In any case,~$(\ref{lem:ixloccst}\,*)$ holds for $\psifstar\cF$.
\end{proof}

\subsection{The \texorpdfstring{$X_0$}{X0}-support condition and duality for \texorpdfstring{$\protect\wt{\ppsif}{}^*\cM$}{M}}\label{subsec:wtpsif*}
Since $\ppsif\pDR\cM$ satisfies the support condition on~$X_0$, the isomorphism \eqref{eq:psipsitilde} shows that $\ppsifstar\cM:=\ppsifstar\pDR\cM$ satisfies the $X_0$\nobreakdash-support condition.

To prove this property for the Verdier dual complex $\bD\ppsifstar\cM$, it~suffices to show that, for a constructible complex $\cF$, there exists an isomorphism $\bD\ppsifstar\cF\isom\wt\ppsif\bD\cF[1]$, since $\pDR$ is compatible with duality. Since $\ppsifstar\cF$ and $\cF$ are $\RR$-constructible, we~have (\cf\cite[Chap.\,3]{K-S90}):
\[
\bD\psifstar\cF\simeq\wtif{}^!\bR\wt j_{f!}\bD\cF\simeq(\psifstar\bD\cF)[-1].
\]
We conclude:
\[
\bD\ppsifstar\cF=\bD(\psifstar\cF[-1])=(\bD\psifstar\cF)[1]\simeq\psifstar\bD\cF\simeq(\ppsifstar\bD\cF)[1].\eqno\qed
\]

\section{The moderate and rapid-decay nearby cycles}\label{sec:pfintro}

\subsection{The functors \texorpdfstring{$\ppsifStar$}{psif}}

We keep the notation of Notation~\ref{nota:cB} and we set
\begin{equation}\label{eq:wtppsimodrd}
\ppsifStar\cM=\wtif{}^{-1}\pDR_{\wt X(f)}^\Star\cM[-1].
\end{equation}

By the faithful flatness of $\iif^{-1}\cO_{\wh{X|X_0}}$ over $\iif^{-1}\cO_X$, for every $\cO_X$-module $\cM$ we~have the equality (\cf \ref{subsec:sheavesblup6} for the notation $\cQ_{X_0}$)
\begin{equation}\label{eq:QXL}
\cQ_{X_0}\overset{\bL}\otimes_{\iif^{-1}\cO_X}\iif^{-1}\cM=\cQ_{X_0}\otimes_{\iif^{-1}\cO_X}\iif^{-1}\cM.
\end{equation}

\begin{lemme}\label{lem:63}
We have
\begin{align*}
\bR\varpif{}_*\wt\ppsif{}^\rmod\cM&\simeq\iif{}^{-1}\pDR_X(\cM(*X_0))[-1],\\
\bR\varpif{}_*\wt\ppsif{}^\rrd\cM&\simeq\pDR(\cQ_{X_0}\otimes\iif^{-1}\cM)[-2].
\end{align*}
Here we regard $\cQ_{X_0}\otimes\iif^{-1}\cM$ as an $\iif^{-1}\cD_X$-module, and
\[
\pDR(\cQ_{X_0}\otimes\iif^{-1}\cM):=\DR(\cQ_{X_0}\otimes\iif^{-1}\cM)[\dim X].\eqno\qed
\]
\end{lemme}

Recall that $\pDR_{\wt X(f)}^\srmod\cM$ and $\pDR_{\wt X(f)}^\srrd\cM$ are supported on $\partial\wt X(f)$. According to~\eqref{eq:DRStarbis} we also have:
\begin{equation}\label{eq:wtppsi>modrd}
\begin{split}
\wt\ppsif{}^\srmod\cM&:=\wtif{}^{-1}\pDR_{\wt X(f)}^\srmod\cM[-1]\simeq\wtif{}^!\pDR_{\wt X(f)}^\rmod\cM,\\
\wt\ppsif{}^\srrd\cM&:=\wtif{}^{-1}\pDR_{\wt X(f)}^\srrd\cM[-1]\simeq\wtif{}^!\pDR_{\wt X(f)}^\rrd\cM.
\end{split}
\end{equation}
Applying the functor $\wtif{}^{-1}[-1]$ to the distinguished triangles \eqref{eq:DRStar} or \eqref{eq:DRStarbis}, we~obtain two distinguished triangles
\begin{equation}\label{eq:distmodrd}
\begin{gathered}
\wt\ppsif{}^\rmod\cM\to\ppsifstar\cM\to\wt\ppsif{}^\srmod\cM\To{+1}\\
\wt\ppsif{}^\rrd\cM\to\ppsifstar\cM\to\wt\ppsif{}^\srrd\cM\To{+1}.
\end{gathered}
\end{equation}

Let $\pi:Y\to X$ be a morphism of complex manifolds and set $g=f\circ\pi$. Recall that there is a natural morphism $\wt\pi=\wt Y(g)\to\wt X(f)$ extending $\pi:Y^*\to X^*$.

\begin{proposition}[Compatibility with projective pushforward]
Assume that $\pi$ is projective. Let $\cM$ be an object of $\catD^\rb_{\pi\textup{-good}}(\cD_Y)$. We~have a functorial isomorphism of distinguished triangles
\[
\xymatrix@R.5cm{
\wt\ppsif{}^\rmod(\Dm\pi_*\cM)\ar[r]\ar[d]^(.4)\wr&\ppsifstar(\Dm\pi_*\cM)\ar[r]\ar[d]^(.4)\wr&\wt\ppsif{}^\srmod(\Dm\pi_*\cM)\ar[r]^-{+1}\ar[d]^(.4)\wr&\\
\bR\wt\pi_*\wt\ppsig{}^\rmod\cM\ar[r]&\bR\wt\pi_*\ppsigstar\cM\ar[r]&\bR\wt\pi_*\wt\ppsig{}^\srmod\cM\ar[r]^-{+1}&
}
\]
and a similar one with rapid decay.
\end{proposition}

\begin{proof}
This is a direct consequence of Corollary~\ref{cor:473}.
\end{proof}

\begin{corollaire}
Assume that $f:X\to\CC$ is projective and let $t$ be a coordinate on $\CC$. Let $\cM$ be a holonomic $\cD_X$\nobreakdash-module. Then the long exact sequence
\begin{multline*}
\cdots\to\cH^k\bR f_*\wt\ppsif{}^\rmod\cM\to\cH^k\bR f_*\ppsifstar\cM\to\cH^k\bR f_*\wt\ppsif{}^\srmod\cM\\
\to\cH^{k+1}\bR f_*\wt\ppsif{}^\rmod\cM\to\cdots
\end{multline*}
splits into short exact sequences, and the short exact sequence
\[
0\to\cH^k\bR f_*\wt\ppsif{}^\rmod\cM\to\cH^k\bR f_*\ppsifstar\cM\to\cH^k\bR f_*\wt\ppsif{}^\srmod\cM\to0
\]
is identified with the short exact sequence
\[
0\to\wt\ppsit{}^\rmod(\Dm f_*^k\cM)\to\wt\ppsit{}^*(\Dm f_*^k\cM)\to\wt\ppsit{}^\srmod(\Dm f_*^k\cM)\to0.
\]
A similar result holds for the rapid-decay complexes. Moreover, $\cH^k\bR f_*\ppsifstar\cM$ is a local system on $S^1$ for each $k$.
\end{corollaire}

\begin{proof}
For a holonomic $\cD_\CC$-module $\cN$, the complex $\wt\ppsit{}^{\!\!\Star}\cN$ can have nonzero cohomology in degree one only, so that $\wt\ppsit{}^{\!\!*}\cN\to\wt\ppsit{}^\srmod\cN$ is surjective. We~conclude that, for each $k$, the $k$th cohomology of the complex $\ppsitStar(\Dm f_*\cM)$ is equal to $\ppsitStar(\Dm f_*^k\cM)$, and the result follows.
\end{proof}

\subsection{Proof of Proposition~\ref{prop:perversite}}\label{pf:perversite}
Assume $\cM$ is holonomic. By \cite[(3.13)]{Kashiwara03} and \eqref{eq:QXL}, we have:
\[
\pDR(\cQ_{X_0}\otimes\iif^{-1}\cM)\simeq\bR\cHom_{\iif^{-1}\cD_X}(\cM^\vee,\cQ_{X_0})[\dim X]\simeq\iif^{-1}\pIrr^*_{X_0}\cM^\vee[1],
\]
by \cite[Cor.\,3.4-4]{Mebkhout04}. Therefore, by Lemma~\ref{lem:63},
\[
\bR\varpif{}_*\wt\ppsif{}^\rrd\cM=\iif^{-1}\pDR(\cQ_{X_0}\otimes\iif^{-1}\cM)[-2]\simeq\iif^{-1}\pIrr^*_{X_0}\cM^\vee[-1].
\]

Similarly, $\bR\varpif{}_*\wt\ppsif{}^\srmod\cM$ is isomorphic to the cone of
\[
\iif^{-1}\pDR\cM(*X_0)\to\iif^{-1}\bR\jf{}_*\jf^{-1}\pDR\cM,
\]
hence is isomorphic to $\iif^{-1}\pIrr_{X_0}\cM$ (\cf\cite[Def.\,3.4-1]{Mebkhout04}).\qed

\subsection{Proof of Theorem~\ref{th:perversite}}\label{subsec:pfperversite}
The $\RR$-constructibility property follows from Theorem~\ref{th:DRmodRconst}, while the case of $\ppsifstar\cM$ was treated in Section~\ref{subsec:wtpsif*}. Furthermore, the $X_0$\nobreakdash-support condition for $\wt\ppsif{}^\rmod\cM[1]$ and $\wt\ppsif{}^\rrd\cM[1]$ follows immediately from Corollary~\ref{cor:X-supp}, from which we deduce that of $\wt\ppsif{}^\rmodrd\cM[1]$. From the distinguished triangles \eqref{eq:distmodrd} we obtain that for $\wt\ppsif{}^\srmod\cM$ and $\wt\ppsif{}^\srrd\cM$.\qed

\subsection{An improvement of Theorem~\ref{th:perversite} in the good case}\label{subsec:goodperversite}
Let $(Y,D)$ be a smooth complex manifold with a normal crossing divisor with smooth components, and let $\cM$ be a good meromorphic flat bundle on $Y$ with poles along $D$. Let $g:Y\to\CC$ be a holomorphic function such that $D_g:=g^{-1}(0)$ is contained in $D$. We~will use the notation as in the proof of Theorem~\ref{th:DRmodRconst}.

\begin{proposition}
The complexes $\wt\ppsig{}^{\rmod D_g}\cM$ and $\wt\ppsig{}^{\rrd D_g}\cM$ satisfy the $X_0$-support condition.
\end{proposition}

\begin{proof}
The statement is local, so~we can use local coordinates $x_1,\dots,x_n$ adapted to~$D$, \ie such that $D=\{x_1\cdots x_\ell=0\}$ and $g(x_1,\dots,x_n)=x_1^{e_1}\cdots x_\ell^{e_\ell}=\bmx^\bme$, with $e_j\geq0$. We~have
\[
\wt Y(D)=(S^1)^\ell\times(\RR_+)^\ell\times\CC^{n-\ell},\quad\partial\wt Y(D)=(S^1)^\ell\times\partial(\RR_+)^\ell\times\CC^{n-\ell},
\]
with coordinates $(\theta_1,\dots,\theta_\ell;\rho_1,\dots,\rho_\ell;x_{\ell+1},\dots,x_n)$, where $\partial(\RR_+)^\ell$ is defined by $\rho_1\cdots\rho_\ell=0$. The map $\varpi_{D,g}$ is given by the formula
\[
(\theta_1,\dots,\theta_\ell;\rho_1,\dots,\rho_\ell;x_{\ell+1},\dots,x_n)\mto\Bigl(\sum e_i\theta_i,\prod\rho_i^{e_i},x_{\ell+1},\dots,x_n\Bigr).
\]
More precisely, if $e_j>0$ for $j=1,\dots,k$ and $e_j=0$ for $j=k+1,\dots,\ell$, then we have
\begin{align*}
\wt Y(D)&\To{\varpi_{D,D_g}}\wt Y(D_g)\\[-5pt]
(\theta_1,\dots,\theta_\ell;\rho_1,\dots,\rho_\ell;\bmx_{\sssup\ell})&\Mto{\hphantom{\varpi_{D,D_g}}}(\theta_1,\dots,\theta_k;\rho_1,\dots,\rho_k;\rho_{\sssup k}e^{i\theta_{\sssup k}},\bmx_{\sssup\ell})\\[5pt]
\tag*{and}
\wt Y(D_g)&\To{\varpi_{D_g,g}}\wt Y(g)\\[-5pt]
(\theta_1,\dots,\theta_k;\rho_1,\dots,\rho_k;\rho_{\sssup k}e^{i\theta_{\sssup k}},\bmx_{\sssup\ell})&\Mto{\hphantom{\varpi_{D_g,g}}}\textstyle\Bigl(\sum_{i=1}^k e_i\theta_i,\prod_{i=1}^k\rho_i^{e_i},\rho_{\sssup k}e^{i\theta_{\sssup k}},\bmx_{\sssup\ell}\Bigr).
\end{align*}
Let $S$ be the stratum defined by $x_1=\cdots=x_\ell=0$ and $x_j\neq0$ for $j>\ell$. Above this stratum, $\partial\wt Y(D)$ is defined by $\rho_1=\cdots=\rho_\ell=0$ and the map $\varpi_{D,g}$ is given by $(\theta_1,\dots,\theta_\ell,\bmx_{\sssup\ell})\mto(\sum_{i=1}^k e_i\theta_i,,\bmx_{\sssup\ell})$, hence its fibers are compact manifolds of real dimension $\ell-1$. As~a consequence, for every sheaf $\cG$ on $\partial\wt Y(D)_{|S}$, we~have $R^j\varpi_{D,g*}\cG=\nobreak0$ for $j>\ell-1$. The $X_0$-support (\ie $D_g$-support) condition then follows from \eqref{eq:Dg} and the generalized Hukuhara-Turrittin theorem \ref{th:drmod}.
\end{proof}

\begin{remarque}
In this case, the support condition for $\pIrr_{D_g}\cM$ and $\pIrr^*_{D_g}\cM$ (\cf \cite[Th.\,3.5-2]{Mebkhout04}) can be obtained as in dimension one, \eg as in \cite[Cor.\,3.16]{Bibi10}, by using \cite[Prop.\,9.23]{Bibi10}.
\end{remarque}

\subsection{Proof of Conjecture~\ref{conj:dualite}}\label{subsec:proofduality}

We first note that, by applying $\iota_f^{-1}$ to the second line of \eqref{eq:>rdrd}, we~obtain a functorial isomorphism
\[
\wt\ppsif{}^\srrd\cM\isom\wtif{}^!\pDR_{\wt X(f)}^\rrd\cM.
\]
Next, we~obtain, by applying $\iota_f^!$ to the isomorphism of Conjecture~\ref{conj:DRdual}\eqref{conj:DRdual1}, that the distinguished triangle
\[
\iota_f^!\bD\pDR_{\wt X(f)}^\srmod\cM\to\iota_f^!\bD\pDR_{\wt X(f)}^*\cM\to\iota_f^!\bD\pDR_{\wt X(f)}^\rmod\cM\To{+1}
\]
is isomorphic to the distinguished triangle
\[
\wtif{}^{-1}\pDR_{\wt X(f)}^\rrd\bD\cM[-1]\to\iota_f^{-1}\bR\wt j_{f*}\,\jf^{-1}\pDR_X\bD\cM[-1]\to\iota_f^!\pDR_{\wt X(f)}^\rrd\bD\cM\To{+1}.
\]
(where we used $\iota_f^!\bR j_{f!}\simeq\iota_f^{-1}\bR j_{f*}[-1]$ for the middle term); using the commutation $\iota_f^!\circ\bD\simeq\bD\circ\iota_f^{-1}$ together with the preliminary remark, we~obtain an isomorphism between the triangle
\[
\bD(\wt\ppsif{}^\srmod\cM[1])\to\bD(\ppsifstar\cM[1])\to\bD(\wt\ppsif{}^\rmod\cM[1])\To{+1}
\]
and the triangle
\[
\wt\ppsif{}^\rrd\bD\cM\to\ppsifstar\bD\cM\to\wt\ppsif{}^\srrd\bD\cM\To{+1}.\eqno\qed
\]

\subsection{Proof of Theorem \ref{th:q0-2}}\label{subsec:pfq0-2}
We begin by exhibiting, for any $\cD_X$-module, a morphism of complexes
\[
\rho_0^{-1}\pDR\bigl(\varpi^{-1}(\wh\Nils\otimes_{\cO_{X\times\CC}}\Dm\gamma_{f*}\cM)\bigr)\to\rho_0^{-1}\pDR^\rmodrd_{X\times\wt\CC}(\Dm\gamma_{f*}\cM),
\]
which is functorial with respect to $\cM$. We use Remark~\ref{rem:gammaf} to identify of the right-hand side with $\rho_0^{-1}\pDR^\rmodrd_{\wt X(f)}(\cM)$. The left-hand side is identified with
\[
\rho_0^{-1}\pDR\bigl([\cA_{X\times\wt\CC}^\nils/\cA_{X\times\wt\CC}^\rrd]\otimes_{\varpi^{-1}\cO_{X\times\CC}}\varpi^{-1}(\Dm\gamma_{f*}\cM)\bigr),
\]
by \eqref{eq:suiteexnilsformel}, and the right-hand side reads
\[
\rho_0^{-1}\pDR\bigl(\cA_{X\times\wt\CC}^\rmodrd\otimes_{\varpi^{-1}\cO_{X\times\CC}}\varpi^{-1}(\Dm\gamma_{f*}\cM)\bigr).
\]
The inclusion $\cA_{X\times\wt\CC}^\nils\hto\cA_{X\times\wt\CC}^\rmod$ provides the desired morphism, even before applying~$\rho_0^{-1}$; proving that it is an isomorphism reduces to showing that
\begin{equation}\label{eq:isonilsmod}
\pDR\bigl(\cA_{X\times\wt\CC}^\nils\otimes_{\varpi^{-1}\cO_{X\times\CC}}\varpi^{-1}(\Dm\gamma_{f*}\cM)\bigr)\to\pDR\bigl(\cA_{X\times\wt\CC}^\rmod\otimes_{\varpi^{-1}\cO_{X\times\CC}}\varpi^{-1}(\Dm\gamma_{f*}\cM)\bigr)
\end{equation}
is an isomorphism. We will thus focus on this assertion.

We proceed by induction on the dimension of the support of $\cM$, the case of dimension zero being clear since both terms vanish. We may thus assume that there exists a projective modification $e:Y\to X$ and a normal crossing divisor $D\subset Y$ such that, setting $g=f\circ e$, we~have $g^{-1}(0)\subset D$ and there exists a good meromorphic flat bundle~$\cM'$ on $Y$ with poles along $D$ such that $\cM=\Dm e_*\cM'$. We recall the diagram~\eqref{eq:blowtilde} and we argue as in the proof of Theorem~\ref{th:DRmodRconst}. Then the desired isomorphism follows from \cite{Mochizuki10}, as already mentioned in Remark~\ref{rem:Nilssonmod}.\qed

\appendix

\section{Base change for a covering map}
Let us consider a Cartesian square of topological spaces:
\[
\xymatrix{
X'\ar[r]^-{g'}\ar[d]_{f'}&X\ar[d]^f\\
Y'\ar[r]_-g\ar@{}[ur]|\square&Y
}
\]

\begin{lemme}
There is a canonical morphism of functors $g^{-1}\circ f_*\to f'_*\circ g^{\prime-1}$.
\end{lemme}

\begin{proof}
The adjunction morphism $\id\to g'_*\circ g^{\prime-1}$ induces a morphism
\[
f_*\to f_*\circ g'_*\circ g^{\prime-1}=g_*\circ f'_*\circ g^{\prime-1}.
\]
We deduce a morphism $g^{-1}\circ f_*\to g^{-1}\circ g_*\circ f'_*\circ g^{\prime-1}$, and using the adjunction $g^{-1}\circ g_*\to\id$, we~obtain the desired morphism.
\end{proof}

For a sheaf $\cG$ on $Y$, we~consider the following property:
\begin{enumerate}
\item[(P)]
Each point $y\in Y$ admits a fundamental system $\mathfrak{V}_y$ of open neighborhoods such that, for each $V\in \mathfrak{V}_y$, the natural morphism $\Gamma(V,\cG)\to\cG_y$ is an isomorphism.\end{enumerate}
We say that a bounded complex $\cG^\cbbullet$ satisfies Property (P) if all its cohomology sheaves do so.

\begin{proposition}\label{prop:basech}
Assume moreover that $f$ is a covering map. Let $\cG^\cbbullet$ be a complex of sheaves on $Y$ satisfying Property \textup{(P)}, as well as $g^{-1}\cG^\cbbullet$ on $Y'$. Then the natural morphism
\begin{starequation}\label{eq:basech}
g^{-1}\circ \bR f_*(f^{-1}\cG^\cbbullet)\to \bR f'_*\circ g^{\prime-1}(f^{-1}\cG^\cbbullet)
\end{starequation}%
is an isomorphism.
\end{proposition}

\begin{proof}
Assume first that $\cG^\cbbullet$ is a sheaf $\cG$. Since the question is local on $Y$, we~may assume that $f$ is the projection $X=Y\times F\to Y$, where $F$ is a discrete set. Then $f^{-1}\cG$ also satisfies property (P) at each point $x\in X$ with $\mathfrak{V}_x=\mathfrak{V}_{f(x)}$. Set $\cF=f^{-1}\cG$ and denote by $\mu:\cF^{\textup{ét}}\to X$ the sheaf space associated with $\cF$. Then Property (P) implies that $f\circ\mu:\cF^{\textup{ét}}\to Y$ is the sheaf space associated with $f_*\cF$.

On the other hand, the sheaf space of $g^{\prime-1}\cF=f^{\prime-1}(g^{-1}\cG)$ is by definition \hbox{$\mu':(g^{\prime-1}\cF)^{\textup{ét}}=X'\bigtimes_X\cF^{\textup{ét}}\to X'$}, and applying Property (P) to $g^{-1}\cG$ we find that the sheaf space of $f'_*g^{-1}\cG$ is $f'\circ\mu':X'\bigtimes_X\cF^{\textup{ét}}\to Y'$. Since we have a Cartesian square, we~identify the latter with $Y'\bigtimes_Y\cF^{\textup{ét}}\to Y'$, as desired.

For an arbitrary bounded complex $\cG^\cbbullet$ as in the statement, we~note that, since $f$ and $f'$ are covering maps, we~have
\begin{align*}
\cH^jg^{-1}\circ \bR f_*(f^{-1}\cG^\cbbullet)&=g^{-1}\circ f_*(f^{-1}\cH^j\cG^\cbbullet),\\
\cH^j\bR f'_*\circ g^{\prime-1}(f^{-1}\cG^\cbbullet)&= f'_*\circ g^{\prime-1}(f^{-1}\cH^j\cG^\cbbullet),
\end{align*}
so we can apply the first part of the proof.
\end{proof}

\begin{exemple}\label{ex:basech}
Assume that $g:Y'\to Y$ is a morphism of real analytic manifolds and that $\cG^\cbbullet$ is weakly $\RR$-constructible (\cf\cite[Def.\,8.4.3]{K-S90}). Then so is $g^{-1}\cG^\cbbullet$, and both satisfy Property (P), by \cite[Prop.\,8.1.4]{K-S90}. Therefore, if $f$ is a covering map, the morphism \eqref{eq:basech} is an isomorphism.
\end{exemple}

\section{Proof of the results in Section~\ref{subsec:sheavesblup}}\label{app:mochizuki}
For the sake of completeness, we~indicate how to use the results of \cite[Chap.\,4\,\&\,5]{Mochizuki10} to obtain those stated in Section~\ref{sec:sheavesblup}.

Let $f:X\to\CC$ be a holomorphic function. Set $X_0=f^{-1}(0)$. Let $\wt\CC$ denote the real blow-up of $\CC$ at the origin. The product $X\times\CC$ is denoted by $\ccX$, the real blowing-up map along $X\times\{0\}$ by $\varpi:\wt\ccX=X\times\wt\CC\to\ccX$ and the open subset $X\times\CC^*$ by $\ccX^*$. For a complex manifold $Y$ with a normal crossing divisor $D_Y$ having smooth components, we~denote by $\wt Y(D_Y)$ the real blow-up of each component of $D_Y$.

The closure of $\gammaf(X^*)$ in $\wt\ccX$ is identified with the real blow-up $\wt X(f)$ of $X$ along $f$. We~have a commutative diagram
\begin{equation}\label{eq:diagblowup}
\begin{array}{c}
\xymatrix@C=1.2cm{
&\wt X(f)\ar@{^{ (}->}[r]^-{\wt\gammaf}\ar[d]^{\varpi_f}&\wt\ccX\ar[d]^\varpi\\
X^*\ar@{^{ (}->}[r]^-{j_f}\ar@{^{ (}->}[ur]^(.5){\wtj_f}
\ar@<-.5ex>@/_2pc/@{^{ (}->}[rrr]^(.53){\gammaf\circ j_f}
&X\ar@{^{ (}->}[r]^-{\gammaf}&\ccX&\ccX^*\ar@{_{ (}->}[l]_(.53){j}\ar@{_{ (}->}[ul]_(.53){\wtj}
}
\end{array}
\end{equation}

In the following, we~consider $\cO_{\gammaf(X)}$ as a sheaf on $X\times\CC$ supported on $\gammaf(X)$. We~also often omit the restriction functor $\gammaf^{-1}$ or the pushforward functor ${\gammaf}_*$ when the context is clear, in order to lighten the notation. The first result clarifies the relation between the sheaves $\cA_{\wt\ccX}^\Star$ and $\cA_{\wt X(f)}^\Star$.

\begin{proposition}[{\cite[Th.\,4.5.1]{Mochizuki10}}]\label{prop:451}
For all values of $\Star$, the following holds:
\begin{enumerate}
\item\label{prop:4511}
The derived tensor product $\varpi^{-1}\cO_{\gammaf(X)}\otimes^{\bL}_{\varpi^{-1}\cO_{\ccX}}\cA_{\wt\ccX}^\Star$ has cohomology only in degree zero, and is supported on $\wt X(f)$.
\item\label{prop:4512}
Let $e:Y\to X$ be a birational morphism which induces an isomorphism \hbox{$Y\moins e^{-1}(X_0)\isom X\moins X_0$} and such that $D_Y:=e^{-1}(X_0)$ has normal crossings with smooth components, and let $e:\wt Y(D_Y)\to\wt\ccX$ be the induced morphism. Then $R^ke_*\cA_{\wt Y(D_Y)}^\Star=0$ for $k>0$ and the natural morphism
\[
\varpi^{-1}\cO_{\gammaf(X)}\otimes_{\varpi^{-1}\cO_{\ccX}}\cA_{\wt\ccX}^\Star\to e_*\cA_{\wt Y(D_Y)}^\Star
\]
induced by $e^*$ is an isomorphism.
\end{enumerate}
\end{proposition}

\begin{proof}
The values $\Star\in\{\rmod,\rrd\}$ are treated in \loccit\ (case $\ell=1$ there), and the value $\Star=*$ is obtained in a very similar way. Proving \eqref{prop:4511} for $\Star\in\{\ssup\rmod,\ssup\rrd,\rmod/\rrd\}$ amounts to proving the injectivity of
\[
\varpi^{-1}\cO_{\gammaf(X)}\otimes_{\varpi^{-1}\cO_{\ccX}}\cA_{\wt\ccX}^\rmod\to\varpi^{-1}\cO_{\gammaf(X)}\otimes_{\varpi^{-1}\cO_{\ccX}}\cA_{\wt\ccX}^*
\]
and similarly for the pairs $(\rrd,*)$ and $(\rrd,\rmod)$. This follows from \eqref{prop:4512} for these pairs. On the other hand, \eqref{prop:4512} for $\Star\in\{\ssup\rmod,\ssup\rrd,\rmod/\rrd\}$ follows directly from \eqref{prop:4512} for $\Star\in\{*,\rmod,\rrd\}$.
\end{proof}

\begin{corollaire}[{\cite[Th.\,4.5.3]{Mochizuki10}}]\label{cor:453}
For all values of $\Star$, the natural morphism
\[
\varpi^{-1}\cO_{\gammaf(X)}\otimes_{\varpi^{-1}\cO_{\ccX}}\cA_{\wt\ccX}^\Star\to\wt\gamma_{f*}\cA_{\wt X(f)}^\Star
\]
is an isomorphism.
\end{corollaire}

\begin{proof}
For $\Star\in\{*,\rmod,\rrd\}$, we~have a natural identification $\rho_*\cA_{\wt Y(D_Y)}^\Star\simeq\wt\gamma_{f*}\cA_{\wt X(f)}^\Star$. For the remaining cases, \eg for $\Star=\ssup\rmod$, $\rho_*\cA_{\wt Y(D_Y)}^{\ssup\rmod}$ is identified with the cokernel of $\rho_*\cA_{\wt Y(D_Y)}^\rmod\to\rho_*\cA_{\wt Y(D_Y)}^*$, according to Proposition~\ref{prop:451}\eqref{prop:4512}, and by definition $\wt\gamma_{f*}\cA_{\wt X(f)}^{\ssup\rmod}$ is the cokernel of $\wt\gamma_{f*}\cA_{\wt X(f)}^\rmod\to\wt\gamma_{f*}\cA_{\wt X(f)}^*$, hence the result.
\end{proof}

Let $\pi:Y\to X$ be a morphism of complex manifolds and set $g=f\circ\pi$. It~can be extended in a unique way as a morphism $\wt\pi:\wt Y(g)\to\wt X(f)$.

\begin{corollaire}[{\cite[Th.\,4.4.3\,\&\,Th.\,4.5.4]{Mochizuki10}}]\label{cor:454}
Let $\pi:Y\to X$ be a projective morphism and let $\cN$ be an inductive limit of coherent $\cO_Y$-modules. Then, for any values of $\Star$, the natural morphism
\[
\cA_{\wt X(f)}^\Star\otimes^{\bL}_{\varpi_f^{-1}\cO_X}\varpi_f^{-1}\bR\pi_*\cN\to\bR\wt\pi_*\bigl(\cA_{\wt Y(g)}^\Star\otimes^{\bL}_{\varpi_g^{-1}\cO_Y}\varpi_g^{-1}\cN\bigr)
\]
is an isomorphism.\qed
\end{corollaire}

\begin{theoreme}[Flatness, {\cite[Th.\,4.6.1]{Mochizuki10}}]
For all values of $\Star$, the sheaves $\cA_{\wt X(f)}^\Star$ are $\varpi_f^{-1}\cO_X$-flat. Furthermore, for any coherent $\cO_X$-module $\cM$, and for $\Star\in\{\rmod,\rrd\}$, the natural morphism
\[
\cA_{\wt X(f)}^\Star\otimes_{\varpi_f^{-1}\cO_X}\varpi_f^{-1}\cM\to\wtj_{f*}\cO_{X^*}\otimes_{\varpi_f^{-1}\cO_X}\varpi_f^{-1}\cM=\wtj_{f*}\cM_{|X^*}
\]
is injective.
\end{theoreme}

\begin{proof}
The theorem is proved in \loccit\ for $\Star\in\{\rmod,\rrd\}$. Flatness in the case \hbox{$\Star=*$} is~easy: it amounts to proving that $\bR\wtj_{f*}\wtj_f^{-1}\cM=\wtj_{f*}\wtj_f^{-1}\cM$ for any coherent $\cO_X$\nobreakdash-module~$\cM$, and this follows from the argument used to prove \eqref{subsec:sheavesblup0} in Section~\ref{subsec:sheavesblup}. Lastly, flatness in the remaining cases, expressed as the vanishing the cohomology of $\cA_{\wt X(f)}^\Star\otimes^{\bL}_{\varpi_f^{-1}\cO_X}\varpi_f^{-1}\cM$ in negative degrees for any coherent $\cO_X$-module $\cM$, follows from both statements in the cases $\Star\in\{*,\rmod,\rrd\}$ by the snake lemma.
\end{proof}

For the following results, we~consider the case of a right $\cD_Y$-module $\cM$ for simplicity and we use the Spencer complex $\Sp$.

\begin{corollaire}[{\cite[Prop.\,4.7.1]{Mochizuki10}}]\label{cor:471}
Setting as in Corollary \ref{cor:454}. For any coherent $\cD_Y$\nobreakdash-module $\cM$ having a coherent filtration (locally with respect to $X$), there is a natural isomorphism in $\catD^\rb(\cD_{\wt X(f)}^\Star)$:
\begin{starequation}\label{eq:471}
\cA_{\wt X(f)}^\Star\otimes^{\bL}_{\varpi_f^{-1}\cO_X}\!\Dm\pi_*\cM\to\bR\wt\pi_*\bigl((\cA_{\wt Y(g)}^\Star\otimes_{\varpi_g^{-1}\cO_Y}\!\varpi_g^{-1}\cM)\otimes_{\varpi_g^{-1}\cD_Y}\!\varpi_g^{-1}\Sp_{Y\to X}\bigr).
\end{starequation}%
\end{corollaire}

\begin{proof}
We can replace $\cM$ by its canonical resolution by induced $\cD_Y$-modules \hbox{$\cM\otimes_{\cO_Y}\Sp_Y$}, so~that we can assume that $\cM=\cN\otimes_{\cO_Y}\cD_Y$, where $\cN$ is an inductive limit of coherent $\cO_Y$-modules, due to the assumption of existence of a coherent filtration. Then $\Dm\pi_*\cM=\bR\pi_*\cN\otimes_{\cO_X}\cD_X$ and the natural morphism
\[
\cA_{\wt X(f)}^\Star\otimes^{\bL}_{\varpi_f^{-1}\cO_X}\varpi_f^{-1}(\Dm\pi_*\cM)\to\bR\wt\pi_*\bigl(\cA_{\wt Y(g)}^\Star\otimes^{\bL}_{\varpi_g^{-1}\cO_Y}\varpi_g^{-1}\cM\bigr)
\]
is an isomorphism by Corollary~\ref{cor:454}. Due to the $\varpi_g^{-1}\cO_Y$-flatness of $\cA_{\wt Y(g)}^\Star$, the latter term is equal to the right-hand side of \eqref{eq:471} with $\cM=\cN\otimes_{\cO_Y}\cD_Y$.
\end{proof}

\begin{corollaire}[{\cite[Cor.\,4.7.3]{Mochizuki10}}]\label{cor:473}
Setting as in Corollary \ref{cor:454}. For any coherent $\cD_Y$-module $\cM$ there is a functorial isomorphism
\[
\pDR^\Star_{\wt X(f)}\Dm\pi_*\cM\simeq\bR\wt\pi_*\pDR^\Star_{\wt Y(g)}\cM.
\]
\end{corollaire}

\begin{proof}
We tensor \eqref{eq:471} on the right by $\varpi_f^{-1}\Sp_X$ and we apply the projection formula to the right-hand side:
\begin{multline*}
\bR\wt\pi_*\Bigl[\cA_{\wt Y(g)}^\Star\otimes_{\varpi_g^{-1}\cO_Y}\varpi_g^{-1}\bigl(\cM\otimes_{\cD_Y}\Sp_{Y\to X}\bigr)\Bigr]\otimes_{\varpi_f^{-1}\cD_X}\varpi_f^{-1}\Sp_X\\
\isom
\bR\wt\pi_*\Bigl[\cA_{\wt Y(g)}^\Star\otimes_{\varpi_g^{-1}\cO_Y}\varpi_g^{-1}\bigl(\cM\otimes_{\cD_Y}\Sp_{Y\to X}\otimes_{\pi^{-1}\cD_X}\pi^{-1}\Sp_X\bigr)\Bigr]\\
\simeq
\bR\wt\pi_*\Bigl[\cA_{\wt Y(g)}^\Star\otimes_{\varpi_g^{-1}\cO_Y}\varpi_g^{-1}\bigl(\cM\otimes_{\cD_Y}\Sp_{Y}\bigr)\Bigr]=\bR\wt\pi_*\pDR^\Star_{\wt Y(g)}\cM.\qedhere
\end{multline*}
\end{proof}

\section{Results about duality}\label{app:C}

We prove a special case of Conjecture~\ref{conj:DRdual}. Let $f,g:X\to\CC$ be holomorphic functions and let $V$ be a meromorphic flat bundle on $X$ with poles along $f^{-1}(0)\cup g^{-1}(0)$. By~definition, we~have $V=V(*f,*g)$. We~denote by $V^\vee$ the dual meromorphic flat bundle.

\begin{proposition}\label{prop:duality}
Locally on $X$, there exist isomorphisms
\begin{align*}
\bD\pDR^\rmod_{\wt X(f)}V&\simeq\pDR^\rrd_{\wt X(f)}V^\vee(!g),\\
\bD\pDR^\rrd_{\wt X(f)}V&\simeq\pDR^\rmod_{\wt X(f)}V^\vee(!g)
\end{align*}
compatible with the natural morphisms from rapid decay to moderate de Rham complexes and inducing the natural isomorphisms existing on $X^*$.
\end{proposition}

\subsection{The case of normal crossing divisors and good meromorphic flat bundle}

Let $(Y,D)$ be a complex manifold with a normal crossing divisor $D$ with a partition of its components in two disjoint sets, giving rise to the decomposition $D=D_1\cup D_2$. Let $V$ be a meromorphic flat bundle with poles along $D$ and let~$V^\vee$ be the dual meromorphic flat bundle. The dual localization $V(!D)$ is defined as $\bD\bigl(V^\vee\bigr)$. We~take up the notation of \cite[\S4.1.4]{Mochizuki10} where ${<}D$ means rapid decay along~$D$ and~${\leq} D$ means of Nilsson class along $D$. We recall that, for a good meromorphic flat bundle $V$ with poles along $D$, the natural morphism from $\DR^{\leq D}V\to\DR^\rmod V$ is a quasi-isomorphism (\cf Remark~\ref{rem:Nilssonmod}). The following de~Rham complexes on~$\wt Y(D)$
\[
\DR_{\wt Y(D)}^{\leq D}(V),\quad\DR_{\wt Y(D)}^{<D_1,\leq D_2}(V),\quad\DR_{\wt Y(D)}^{<D_2,\leq D_1}(V),\quad\DR_{\wt Y(D)}^{<D}(V)
\]
enter a natural commutative diagram
\begin{equation}\label{eq:diagD}
\begin{array}{c}
\xymatrix{
\DR_{\wt Y(D)}^{<D}(M)\ar[r]\ar[d]&\DR_{\wt Y(D)}^{<D_2,\leq D_1}(M)\ar[d]\\
\DR_{\wt Y(D)}^{<D_1,\leq D_2}(M)\ar[r]&\DR_{\wt Y(D)}^{\leq D}(M)
}
\end{array}
\end{equation}

The dual diagram $\bD\eqref{eq:diagD}$ is obtained by taking the Verdier dual at each vertex and considering the dual arrows, which then point toward the opposite direction. In~order to obtain a diagram similar to \eqref{eq:diagD}, it~is thus necessary to consider first a diagram $\eqref{eq:diagD}^\perp$ obtained from \eqref{eq:diagD} by reflection through to its center. Then $\bD\eqref{eq:diagD}^\perp$ is
\begin{equation}
\tag{$\bD\eqref{eq:diagD}^\perp$}\label{eq:diagDD}
\begin{array}{c}
\xymatrix{
\bD\DR_{\wt Y(D)}^{\leq D}(V)\ar[r]\ar[d]&\bD\DR_{\wt Y(D)}^{<D_1,\leq D_2}(V)\ar[d]\\
\bD\DR_{\wt Y(D)}^{<D_2,\leq D_1}(V)\ar[r]&\bD\DR_{\wt Y(D)}^{<D}(V)
}
\end{array}
\end{equation}
There is a natural morphism of squares
\begin{equation}\label{eq:C1C1perp}
\eqref{eq:diagD}(V^\vee)\to\eqref{eq:diagDD}(V),
\end{equation}
meaning that there is a natural morphism between the corresponding vertices (in~$\catD^\rb(\CC_{\wt Y(D)})$) and these morphisms are compatible with the arrows in the squares.

\begin{theoreme}[{\cite[Th.\,5.2.2]{Mochizuki10}}]
If $V$ is a good meromorphic flat bundle along $(Y,D)$, the morphism \eqref{eq:C1C1perp} is an isomorphism.
\end{theoreme}

\subsection{De Rham complexes on $\wt Y(D_1)$}\label{subsec:Cb}
We now consider the real blow-up $\varpi_1:\wt Y(D_1)\to Y$ and the commutative diagram of morphisms
\[
\xymatrix{
\wt Y(D)\ar[rr]^{\wt\varpi}\ar[rd]_\varpi&&\wt Y(D_1)\ar[ld]^{\varpi_1}\\
&Y&
}
\]
We also consider the following de~Rham complexes on $\wt Y(D_1)$:
\[
\DR_{\wt Y(D_1)}^{\leq D_1}(V),\quad\DR_{\wt Y(D_1)}^{<D_1}(V),\quad\DR_{\wt Y(D_1)}^{\leq D_1}(V(!D_2)),\quad\DR_{\wt Y(D_1)}^{<D_1}(V(!D_2)).
\]
There is a natural commutative diagram
\begin{equation}\label{eq:diagD1}
\begin{array}{c}
\xymatrix{
\DR_{\wt Y(D_1)}^{<D_1}(V(!D_2))\ar[r]\ar[d]&\DR_{\wt Y(D_1)}^{\leq D_1}(V(!D_2))\ar[d]\\
\DR_{\wt Y(D_1)}^{<D_1}(V)\ar[r]&\DR_{\wt Y(D_1)}^{\leq D_1}(V)
}
\end{array}
\end{equation}

\begin{proposition}\label{prop:C6}
If $V$ is a good meromorphic flat bundle along $(Y,D)$, there is a functorial isomorphism (in the derived category) $\eqref{eq:diagD1}\to\bR\wt\varpi_*\eqref{eq:diagD}$ (\ie a morphism between the corresponding vertices which is compatible with the arrows).
\end{proposition}

\begin{proof}

\begin{lemme}\label{lem:identrho}
We have the following identifications:
\begin{align*}
\bR\wt\varpi_*\cA_{\wt Y(D)}^{\leq D}&\simeq\cA_{\wt Y(D_1)}^{\leq D_1}(*D_2),\\
\bR\wt\varpi_*\cA_{\wt Y(D)}^{<D_1,\leq D_2}&\simeq\cA_{\wt Y(D_1)}^{<D_1}(*D_2),\\
\bR\wt\varpi_*\cA_{\wt Y(D)}^{<D_2,\leq D_1}&\simeq\bigl\{\cA_{\wt Y(D_1)}^{\leq D_1}\to\cA_{\wt Y(D_1)|\wh D_2}^{\leq D_1}\bigr\},\\
\bR\wt\varpi_*\cA_{\wt Y(D)}^{<D}&\simeq\bigl\{\cA_{\wt Y(D_1)}^{< D_1}\to\cA_{\wt Y(D_1)|\wh D_2}^{< D_1}\bigr\},
\end{align*}
where, denoting by $f_2$ a local defining function of $D_2$, we~have set
\[
\cA_{\wt Y(D_1)|\wh D_2}^{\leq D_1}:=\varprojlim_k\cA_{\wt Y(D_1)}^{\leq D_1}\big/f_2^k\cA_{\wt Y(D_1)}^{\leq D_1},\quad\cA_{\wt Y(D_1)|\wh D_2}^{< D_1}:=\varprojlim_k\cA_{\wt Y(D_1)}^{< D_1}\big/f_2^k\cA_{\wt Y(D_1)}^{< D_1}.
\]
\end{lemme}

\begin{proof}
The first two identifications follow from \cite[Th.\,4.3.2 \& Th.\,4.3.1]{Mochizuki10} and the last two from the same references together with \cite[Lem.\,II.1.1.18]{Bibi97}.
\end{proof}

We denote by $\square$ one of the symbols ${<}D,\dots,{\leq} D$ appearing in the diagram \eqref{eq:diagD}. We~know from \cite[Th.\,4.3.1]{Mochizuki10} that $\cA_{\wt Y(D)}^\ssquare$ is $\varpi^{-1}\cO_Y$-flat. For any holonomic $\cD_Y$-module $M$, we~have an isomorphism
\begin{multline*}
\bigl(\bR\wt\varpi_*\cA_{\wt Y(D)}^\ssquare\otimes^{\bL}_{\varpi_1^{-1}\cO_Y}\varpi_1^{-1}M\bigr)\otimes_{\varpi_1^{-1}\cD_Y}\varpi_1^{-1}\Sp_Y\\
\isom
\bR\wt\varpi_*\bigl(\cA_{\wt Y(D)}^\ssquare\otimes^{\bL}_{\varpi^{-1}\cO_Y}\varpi^{-1}M\bigr)\otimes_{\varpi_1^{-1}\cD_Y}\varpi_1^{-1}\Sp_Y\\
=
\bR\wt\varpi_*\bigl(\cA_{\wt Y(D)}^\ssquare\otimes_{\varpi^{-1}\cO_Y}\varpi^{-1}M\bigr)\otimes_{\varpi_1^{-1}\cD_Y}\varpi_1^{-1}\Sp_Y\\
\isom
\bR\wt\varpi_*\Bigl[\bigl(\cA_{\wt Y(D)}^\ssquare\otimes_{\varpi^{-1}\cO_Y}\varpi^{-1}M\bigr)\otimes_{\varpi^{-1}\cD_Y}\varpi^{-1}\Sp_Y\Bigr].
\end{multline*}

We denote by $\cA_{\wt Y(D_1)}^\ssquare$ the object defined by the right-hand sides in Lemma~\ref{lem:identrho}. From the first two lines of Lemma~\ref{lem:identrho} together with flatness properties (that also hold for $\cA_{\wt Y(D_1)|\wh D_2}^{\leq D_1}$ and $\cA_{\wt Y(D_1)|\wh D_2}^{< D_1}$ according to \loccit), the first line above reads
\[
\bigl(\cA_{\wt Y(D_1)}^\ssquare\otimes_{\varpi_1^{-1}\cO_Y}\varpi_1^{-1}M\bigr)\otimes_{\varpi_1^{-1}\cD_Y}\varpi_1^{-1}\Sp_Y.
\]
Let us consider the square
\begin{equation}\tag{\ref{eq:diagD1}$'$}\label{eq:diagD1prime}
\begin{array}{c}
\xymatrix{
\DR_{\wt Y(D_1)}^{<D}(M)\ar[r]\ar[d]&\DR_{\wt Y(D_1)}^{<D_2,\leq D_1}(M)\ar[d]\\
\DR_{\wt Y(D_1)}^{<D_1}(M(*D_2))\ar[r]&\DR_{\wt Y(D_1)}^{\leq D_1}(M(*D_2))
}
\end{array}
\end{equation}
Then for any holonomic $\cD_Y$-module $M$, we~have a functorial isomorphism $\eqref{eq:diagD1prime} \isom\bR\wt\varpi_*\eqref{eq:diagD}$.

Let us now assume that $M=M(!D_2)$. Then there is a natural morphism $\eqref{eq:diagD1prime}\to\eqref{eq:diagD1}$ which is the identity on the lower lines. It~is induced by the morphism of complexes (for $\square={}<\!D_1,\leq\!D_1$)
\[
\{\cA_{\wt Y(D_1)}^\ssquare\otimes \varpi_1^{-1}M\ra\cA_{\wt Y(D_1)|\wh D_2}^\ssquare\otimes \varpi_1^{-1}M\bigr\}\to\cA_{\wt Y(D_1)}^\ssquare\otimes \varpi_1^{-1}M.
\]
The proposition then follows from Lemma~\ref{lem:vanishing} below.
\end{proof}

\begin{lemme}[\cite{Mochizuki10}]\label{lem:vanishing}
Assume that $V$ is a good meromorphic flat bundle along $(Y,D)$. Then, for $\square={}<\!D_1,\leq\!D_1$, we~have $\DR^\ssquare_{\wt Y(D_1)|\wh D_2}(V(!D_2))=0$.
\end{lemme}

\begin{proof}
This statement is proved in \cite[Prop.\,3.2.2]{Mochizuki10} when replacing $\wt Y(D_1)$ with~$Y$. We~adapt an earlier proof of the latter statement which appeared in a preliminary version of~\cite{Mochizuki10}. A similar proof is given in Lemmas 5.1.6 and 5.1.8 of \loccit

Let us set $D_2=\bigcup_{i\in J}D_i$ ($J\neq\emptyset$) and, for any nonempty subset $I\subset J$, let us set $D_I=\bigcap_{i\in I}D_i$. By~a Mayer-Vietoris argument (\cf \cite[Lem.\,II.1.1.13]{Bibi97}), it~is enough to prove the vanishing $\DR^\ssquare_{\wt Y(D_1)|\wh D_I}(V(!D_2))=0$. This is a local question, so~that we can use a local coordinate system $(x_1,\dots,x_n)$ adapted to $D$. It~is then enough to prove that, for some $i\in I$, the differential
\[
\partial_{x_i}:\cA_{\wt Y(D_1)|\wh D_I}^\ssquare\otimes_{\varpi_1^{-1}\cO_Y}\varpi_1^{-1}(V(!D_2))\to\cA_{\wt Y(D_1)|\wh D_I}^\ssquare\otimes_{\varpi_1^{-1}\cO_Y}\varpi_1^{-1}(V(!D_2))
\]
is bijective. In~the following, we~fix such an $i$. We~have an identification
\[
\cA_{\wt Y(D_1)|\wh D_I}^\ssquare\otimes_{\varpi_1^{-1}\cO_Y}\varpi_1^{-1}(V(!D_2))=\cA_{\wt Y(D_1)|\wh D_I}^\ssquare\otimes_{\varpi_1^{-1}\cO_{Y|\wh D_I}}\varpi_1^{-1}(\cO_{Y|\wh D_I}\otimes_{\varpi_1^{-1}\cO_Y}V(!D_2)),
\]
and by goodness, $\cO_{Y|\wh D_I}\otimes_{\varpi_1^{-1}\cO_Y}V$ locally decomposes, possibly after a finite ramification around $D$, into a direct sum of terms obtained from regular holonomic $\cD_{Y|\wh D_I}(*D_2)$-modules $R$ by twisting their connection by $\rd\varphi$ for some good local section~$\varphi$ of $\cO_Y(*D)/\cO_Y$ (\ie $\varphi$ is a the product of a monomial in $x_1,\dots,x_n$ with negative exponents by a unit in $\cO_Y$). Since $R$ is a successive extension of rank-one objects, it~is enough to assume that $R$ has rank one by an easy induction. We~can thus assume that $V=(\cO_Y(*D),\rd+\rd\varphi+\omega$, where $\omega$ is a logarithmic $1$-form with constant coefficients. Then the computation is standard (such computations are done in \cite[Proof of Lem.\,5.1.6]{Mochizuki10}).
\end{proof}

\begin{corollaire}\label{cor:vanishing}
There exists an isomorphism of squares from\enlargethispage{\baselineskip}
\[
\xymatrix@R=.5cm{
\DR_{\wt Y(D_1)}^{<D_1}(V^\vee(!D_2))\ar[r]\ar[d]&\DR_{\wt Y(D_1)}^{\leq D_1}(V^\vee(!D_2))\ar[d]\\
\DR_{\wt Y(D_1)}^{<D_1}(V^\vee)\ar[r]&\DR_{\wt Y(D_1)}^{\leq D_1}(V^\vee)
}
\]
to
\[
\xymatrix@R=.5cm{
\bD\DR_{\wt Y(D_1)}^{\leq D_1}(V)\ar[r]\ar[d]&\bD\DR_{\wt Y(D_1)}^{<D_1}(V)\ar[d]\\
\bD\DR_{\wt Y(D_1)}^{\leq D_1}(V(!D_2))\ar[r]&\bD\DR_{\wt Y(D_1)}^{<D_1}(V(!D_2))
}
\]
which extends the natural isomorphisms existing on $Y^*:=Y\moins D_1$.
\end{corollaire}

\begin{proof}
We first apply $\bR\wt\varpi_*$ to the isomorphism \eqref{eq:C1C1perp}. By~applying the isomorphism of Proposition~\ref{prop:C6} together with the commutation isomorphism $\bR\wt\varpi_*\bD\simeq\bD\bR\wt\varpi_*$, we~obtain the desired isomorphism. We now check compatibility with the natural isomorphisms when restricted to $Y^*$. We~note that $\wt\varpi:\wt Y^*(D_2)\to Y^*$ is the real blow-up of the components of $D_2$. We~apply \cite[Prop.\,5.2.1]{Mochizuki10} to this real blow-up (denoted by $\pi$ in \loccit) and with $D=D_2$.
\end{proof}

\subsection{End of the proof of Proposition~\ref{prop:duality}}
The question is local on $X$. Let \hbox{$\pi:Y\to X$} be a projective modification with $Y$ smooth such that $fg\circ\pi$ defines a divisor with normal crossings $D$. We~denote by $D_1$ the divisor defined by $f\circ\pi$ and by $D_2$ the union of the remaining components of $D$. There exists a unique lift $\wt\pi:\wt Y(D_1)\to\wt X(f)$ of $\pi$. Let $V$ be a meromorphic flat bundle on $X$ with poles along $(fg)^{-1}(0)$ and let $V_Y$ denote its pullback on $Y$, which has poles along~$D$. We~have $\Dm\pi_*V_Y=V$ (this is seen for example by computing the pushforward by means of induced $\cD$-modules). Since duality commutes with pushforward, we~deduce that $\Dm\pi_*(V_Y(!g\circ\pi))\simeq V(!g)$. On the other hand, one checks that \hbox{$V_Y(!g\circ\pi)(*D_1)=V_Y(!D_2)$}, so~that, when considering the de~Rham complexes of Corollary~\ref{cor:vanishing} one can replace $V_Y(!g\circ\pi)$ with $V_Y(!D_2)$. Finally, applying $\bR\varpi_*$ to the isomorphism of Corollary~\ref{cor:vanishing} together with the isomorphism of Corollary~\ref{cor:473} (stated for $\wt Y(f\circ\pi)$, but which also applies to $\wt Y(D_1)$), we~obtain the isomorphisms of Proposition~\ref{prop:duality}.

Finally, the compatibility statement is a consequence of the one obtained in Corollary~\ref{cor:vanishing}.\qed

\backmatter
\bibliographystyle{smfalpha-cs}
\bibliography{smfjournalnames,sabbah}

\providecommand{\SortNoop}[1]{}\providecommand{\sortnoop}[1]{}\providecommand{\eprint}[1]{\href{http://arxiv.org/abs/#1}{\texttt{arXiv\string:\allowbreak#1}}}\providecommand{\hal}[1]{\href{https://hal.archives-ouvertes.fr/hal-#1}{\texttt{hal-#1}}}\providecommand{\tel}[1]{\href{https://hal.archives-ouvertes.fr/tel-#1}{\texttt{tel-#1}}}\providecommand{\doi}[1]{\href{http://dx.doi.org/#1}{\texttt{doi\string:\allowbreak#1}}}\providecommand{\didotfam}{}
\providecommand{\bysame}{\leavevmode ---\ }
\providecommand{\og}{``}
\providecommand{\fg}{''}
\providecommand{\smfandname}{\&}
\providecommand{\smfedsname}{\'eds.}
\providecommand{\smfedname}{\'ed.}
\providecommand{\smfmastersthesisname}{M\'emoire}
\providecommand{\smfphdthesisname}{Th\`ese}
\begin{thebibliography}{Meb04}

\bibitem[DK16]{D-K13}
{\scshape A.~D'Agnolo {\normalfont \smfandname} M.~Kashiwara} -- {\og
  Riemann-{H}ilbert correspondence for holonomic {D}-modules\fg}, \emph{Publ.
  Math. Inst. Hautes {\'E}tudes Sci.} \textbf{123} (2016), p.~69--197.

\bibitem[Del73]{Deligne73}
{\scshape P.~Deligne} -- {\og {Le formalisme des cycles {\'e}vanescents
  (expos{\'e}s 13 et 14)}\fg}, in \emph{{SGA~7\,II}}, Lect. Notes in Math.,
  vol. 340, Springer-Verlag, 1973, p.~82--173.

\bibitem[HH25]{H-H25}
{\scshape B.~Hepler {\normalfont \smfandname} A.~Hohl} -- {\og Moderate growth
  and rapid decay nearby cycles via enhanced ind-sheaves\fg}, \emph{Publ. RIMS,
  Kyoto Univ.} \textbf{61} (2025), no.~1, p.~1--51.

\bibitem[Hie07]{Hien07}
{\scshape M.~Hien} -- {\og {Periods for irregular singular connections on
  surfaces}\fg}, \emph{Math. Ann.} \textbf{337} (2007), p.~631--669.

\bibitem[Kas03]{Kashiwara03}
{\scshape M.~Kashiwara} -- \emph{{$D$}-modules and microlocal calculus},
  Translations of Mathematical Monographs, vol. 217, American Mathematical
  Society, Providence, RI, 2003.

\bibitem[KS90]{K-S90}
{\scshape M.~Kashiwara {\normalfont \smfandname} P.~Schapira} -- \emph{{Sheaves
  on Manifolds}}, Grundlehren Math. Wissen., vol. 292, Springer-Verlag, Berlin,
  Heidelberg, 1990.

\bibitem[Ked11]{Kedlaya10}
{\scshape K.~Kedlaya} -- {\og Good formal structures for flat meromorphic
  connections,~{II}: excellent schemes\fg}, \emph{J.~Amer. Math. Soc.}
  \textbf{24} (2011), no.~1, p.~183--229.

\bibitem[Maj84]{Majima84}
{\scshape H.~Majima} -- \emph{Asymptotic analysis for integrable connections
  with irregular singular points}, Lect. Notes in Math., vol. 1075,
  Springer-Verlag, 1984.

\bibitem[Mal85]{malgrange85}
{\scshape B.~Malgrange} -- {\og Sur les images directes de
  {$\mathcal{D}$}-modules\fg}, \emph{Manuscripta Math.} \textbf{50} (1985),
  p.~49--71.

\bibitem[Mal91]{Malgrange91}
\bysame , \emph{{\'E}quations diff{\'e}rentielles {\`a} coefficients
  polynomiaux}, Progress in Math., vol.~96, Birkh{\"a}user, Basel, Boston,
  1991.

\bibitem[Meb90]{Mebkhout90}
{\scshape Z.~Mebkhout} -- {\og Le th{\'e}orème de positivit{\'e} de
  l'irr{\'e}gularit{\'e} pour les {$\mathcal{D}_{X}$}-modules\fg}, in
  \emph{{The Grothendieck Festschrift}}, Progress in Math., vol. 88, no.~3,
  Birkh{\"a}user, Basel, Boston, 1990, p.~83--132.

\bibitem[Meb04]{Mebkhout04}
\bysame , {\og {Le th{\'e}orème de positivit{\'e}, le th{\'e}orème de
  comparaison et le th{\'e}orème d'existence de Riemann}\fg}, in
  \emph{{{\'E}l{\'e}ments de la th{\'e}orie des systèmes diff{\'e}rentiels
  g{\'e}om{\'e}triques}}, S{\'e}minaires \& Congrès, vol.~8, Soci{\'e}t{\'e}
  Math{\'e}matique de France, Paris, 2004, p.~165--310.

\bibitem[MS89]{M-S86b}
{\scshape Z.~Mebkhout {\normalfont \smfandname} C.~Sabbah} -- {\og {\S\kern
  .15em III.4 {$\mathcal{D}$}-modules et cycles {\'e}vanescents}\fg}, in
  \emph{{Le formalisme des six op{\'e}rations de Grothendieck pour les
  {$\mathcal{D}$}\nobreakdash-mo\-du\-les coh{\'e}rents}}, Travaux en cours,
  vol.~35, Hermann, Paris, 1989, p.~201--239.

\bibitem[Moc09]{Mochizuki09}
{\scshape T.~Mochizuki} -- {\og {On Deligne-Malgrange lattices, resolution of
  turning points and harmonic bundles}\fg}, \emph{Ann. Inst. Fourier
  (Grenoble)} \textbf{59} (2009), no.~7, p.~2819--2837.

\bibitem[Moc11]{Mochizuki08}
\bysame , \emph{{\SortNoop{1}Wild harmonic bundles and wild pure twistor
  $D$-modules}}, Ast{\'e}\-risque, vol. 340, Soci{\'e}t{\'e} Math{\'e}matique
  de France, Paris, 2011.

\bibitem[Moc14]{Mochizuki10}
\bysame , \emph{{Holonomic $\mathcal D$-modules with Betti structure}},
  M{\'e}m. Soc. Math. France (N.S.), vol. 138--139, Soci{\'e}t{\'e}
  Math{\'e}matique de France, Paris, 2014.

\bibitem[Sab93]{Bibi93}
{\scshape C.~Sabbah} -- {\og {{\'E}quations diff{\'e}rentielles {\`a} points
  singuliers irr{\'e}guliers en dimension~$2$}\fg}, \emph{Ann. Inst. Fourier
  (Grenoble)} \textbf{43} (1993), p.~1619--1688.

\bibitem[Sab00]{Bibi97}
\bysame , \emph{{{\'E}quations diff{\'e}rentielles {\`a} points singuliers
  irr{\'e}guliers et ph{\'e}\-no\-mène de Stokes en dimension~{$2$}}},
  Ast{\'e}\-risque, vol. 263, Soci{\'e}t{\'e} Math{\'e}matique de France,
  Paris, 2000.

\bibitem[Sab13]{Bibi10}
\bysame , \emph{{\SortNoop{1}Introduction to Stokes structures}}, Lect. Notes
  in Math., vol. 2060, Springer-Verlag, 2013.

\bibitem[Ser65]{Serre65}
{\scshape J.-P. Serre} -- \emph{Algèbre locale et multiplicit{\'e}s}, 3\ieme
  \smfedname, Lect. Notes in Math., vol.~11, Springer-Verlag, 1965.

\bibitem[Tey16]{Teyssier16b}
{\scshape J.-B. Teyssier} -- {\og A boundedness theorem for nearby slopes of
  holonomic {$\mathcal{D}$}-modules\fg}, \emph{Compositio Math.} \textbf{152}
  (2016), no.~10, p.~2050--2070.

\end{thebibliography}
\end{document}